\documentclass[12pt]{amsart}
\usepackage{ amsmath,amssymb,amsthm,color, euscript, enumerate}
\usepackage{dsfont}
\newcommand{\cds}{\cdots}

\newcommand{\1}{ \mathds{1}}

\newcommand{\vsb}{\vspace{2mm}}

\newcommand{\q}{\quad}

\newcommand{\maru}[1]{{\ooalign{\hfil#1\/\hfil\crcr
\raise.167ex\hbox{\mathhexbox20D}}}}

\newcommand{\ruby}[2]{%
 \leavevmode
 \setbox0=\hbox{#1}%
 \setbox1=\hbox{\tiny #2}%
 \ifdim\wd0>\wd1 \dimen0=\wd0 \end{lemma}se \dimen0=\wd1 \fi
 \hbox{%
   \kanjiskip=0pt plus 2fil
   \xkanjiskip=0pt plus 2fil
   \vbox{%
     \hbox to \dimen0{%
       \tiny \hfil#2\hfil}%
     \nointerlineskip
     \hbox to \dimen0{\mathstrut\hfil#1\hfil}}}}

\newcommand{\la}{\langle}
\newcommand{\ra}{\rangle}

\DeclareMathOperator*{\tensor}{\otimes}
\DeclareMathOperator*{\fusion}{\boxtimes}
\newcommand{\Z}{\mathbb{Z}}
\newcommand{\C}{\mathbb{C}}
\newcommand{\R}{\mathbb{R}}

\newcommand{\RM}{\mathrm{RM}}

\newcommand{\g}{\mathfrak{g}}

\newcommand{\F}{\mathbb{F}}

\newcommand{\Sym}{{\rm Sym}}
\newcommand{\Imm}{{\rm Im}}
\newcommand{\Ker}{{\rm Ker}}
\newcommand{\vir}{\mathrm{Vir}}
\newcommand{\aut}{\mathrm{Aut}\,}
\newcommand{\Aut}{\mathrm{Aut}\,}
\newcommand{\Inn}{\mathrm{Inn}\,}
\newcommand{\wt}{\mathrm{wt}}

\newcommand{\be}{\beta}
\newcommand{\al}{\alpha}
\newcommand{\EuD}{\EuScript{D}}
\makeatletter \@addtoreset{equation}{section}

\theoremstyle{plain}
\newtheorem{theorem}{Theorem}[section]
\newtheorem{proposition}[theorem]{Proposition}
\newtheorem{lemma}[theorem]{Lemma}
\newtheorem{corollary}[theorem]{Corollary}

\theoremstyle{definition}
\newtheorem{definition}[theorem]{Definition}

\newtheorem{notation}[theorem]{Notation}
\theoremstyle{remark}
\newtheorem{remark}[theorem]{Remark}
\numberwithin{equation}{section}

\title[Classification of Holomorphic framed VOAs]{Classification of holomorphic framed vertex operator algebras of central charge 24}
\author{Ching Hung Lam} %
  \address[C. H. Lam] {Institute of Mathematics, Academia Sinica, Taipei 10617, Taiwan and National Center for Theoretical Sciences of  Taiwan.}
  \email{chlam@math.sinica.edu.tw}
\author[H. Shimakura]{Hiroki Shimakura}%
\address[H. Shimakura]{Graduate School of Information Sciences,
Tohoku University,
Aramaki aza Aoba 6-3-09, Aoba-ku Sendai-city, 980-8579, Japan }%
\email {shimakura@m.tohoku.ac.jp}%
\date{}
\thanks{C.\,H. Lam was partially supported by NSC grant
  100-2628-M-001005-MY4 of Taiwan}
\thanks{H.\ Shimakura was partially supported by Grants-in-Aid for Scientific Research (No. 23540013), JSPS}

\newcommand{\sfr}[2]{\leavevmode\kern-.1em
  \raise.5ex\hbox{\the\scriptfont0 #1}\kern-.1em
  /\kern-.15em\lower.25ex\hbox{\the\scriptfont0 #2}}

\newcommand{\shf}{\sfr{1}{2}}

\newcommand{\supp}{\mathrm{supp}}

\begin{document}

\begin{abstract}
This article is a continuation of our work on the classification of holomorphic framed vertex operator algebras of central charge $24$.
We show that a holomorphic framed VOA of central charge $24$ is uniquely determined by the Lie algebra structure of its
weight one subspace. As a consequence, we completely classify all holomorphic framed vertex operator algebras of central charge $24$ and show that there exist exactly $56$ such vertex operator algebras, up to isomorphism.
\end{abstract}
\maketitle

%\tableofcontents

\section{Introduction}

The classification of holomorphic vertex operator algebras (VOAs) of central charge $24$ is  one of the fundamental problems in
vertex operator algebras and mathematical physics.  In 1993 Schellekens \cite{Sc93} obtained  a partial classification by
determining possible Lie algebra structures  for the weight one subspaces  of holomorphic VOAs of  central charge $24$. There are
$71$ cases in his list but only $39$ of the $71$ cases were known explicitly at that time. It is also an open question if the Lie
algebra structure of the weight one subspace will determine the VOA structure uniquely when the central charge is $24$.
Recently, a special class of holomorphic VOAs, called framed VOAs, was studied in \cite{Lam, LS}. Along with other results,  $17$
new examples were constructed. Moreover, it was shown in \cite{Lam,LS} that there exist exactly $56$ possible Lie algebras for
holomorphic framed VOAs of central charge $24$ and all cases can be constructed explicitly. In this article, we complete the
classification of holomorphic framed VOAs of central charge $24$. The main theorem is as follows:

\begin{theorem} \label{MainT} The isomorphism class of a holomorphic framed VOA of central charge $24$ is
uniquely determined by the Lie algebra structure  of its weight one subspace. In particular, there exist exactly $56$ holomorphic
framed VOAs of central charge $24$, up to isomorphism.
\end{theorem}

\begin{remark}
By our classification\,(see \cite[Table 1]{LS}), we noticed that the levels of the representations of Lie algebra associated to the
weight one subspace are powers of two for any holomorphic framed VOA of central charge $24$. Conversely, by comparing with
the list of Lie algebras in \cite{Sc93}, we found that except for one case where the Lie algebra has the type
$E_{6,4}C_{2,1}A_{2,1}$, all other Lie algebras in \cite{Sc93} can be  obtained from holomorphic framed VOAs if the levels are powers of two.
\end{remark}

First let us recall the results in \cite{Lam,LS} and discuss our methods. It was shown in \cite{LY} that a  code $D$  of length divisible
by $16$ can be realized as a $\sfr{1}{16}$-code of a holomorphic framed VOA if and only if $D$ is triply even and the all-one vector
$\mathbf{1}\in D$. Therefore, the classification of holomorphic framed VOAs of rank $8k$ can be reduced into the following 2
steps:
\begin{enumerate}[(1)]
\item classify all triply even codes $D$ of length $16k$ such that $\mathbf{1}\in D$;
\item determine all possible VOA structures with the $\sfr{1}{16}$-code $D$ for each triply even code $D$;
\end{enumerate}

\begin{notation}\label{doubling}
Let $E$ be a doubly even code of length $n$ and let $d: \Z_2^n \to \Z_2^{2n}$ be the linear map defined by $d(\al)=(\al, \al)$. The code
$\EuD(E)= \la d(E), (\mathbf{1}, 0)\ra_{\Z_2}$ spanned by $d(E)$ and $(\mathbf{1}, 0)$ is called  the \textit{extended doubling} of $E$,
where $\mathbf{1}$ is the all-one vector.

Let $\RM(1,4)$ be the first order Reed-Muller code of degree $4$ and  $d_{16}^+$  the unique indecomposable doubly even
self-dual code of length $16$. We also use $A\oplus B$ to denote the direct sum of two subcodes $A$ and $B$.
\end{notation}

Recently, triply even codes of length $48$ were classified by Betsumiya-Munemasa \cite[Theorem 29]{BM}: a maximal triply even
code of length $48$ is isomorphic to an extended doubling, a direct sum of extended doublings or an exceptional triply even code
$D^{ex}$ of dimension $9$.
By this result,
the classification of holomorphic framed VOAs of central charge $24$ can be divided into the following 4 cases. Let $D$ be a $\sfr{1}{16}$-code of a holomorphic framed VOA $U$ of central charge $24$.  Then, up to equivalence,
\begin{enumerate}[(i)]
\item $D$ is subcode of an extended doubling $\EuD(E)$ for some doubly even code $E$ of length $24$;
\item $D$ is a subcode of $\RM(1,4)^{\oplus 3}$  but is not contained in an extended doubling;
\item  $D$ is a subcode of ${\rm RM}(1,4) \oplus \EuD(d_{16}^+)$ but  is not contained in an extended doubling or $\RM(1,4)^{\oplus 3}$;
\item  $D$ is a subcode of the $9$-dimensional exceptional triply even code $D^{ex}$ of length $48$
but is not contained in an extended doubling, $\RM(1,4)^{\oplus3}$ or ${\rm RM}(1,4)\oplus \EuD(d_{16}^+)$.
\end{enumerate}

\medskip
The main idea is to enumerate all possible framed VOA structures in each case.

\noindent\textbf{Case (i).}  If $D$ is a subcode of  an extended doubling, then it was shown \cite[Theorem 3.9]{Lam} that $U$ is
isomorphic to a lattice VOA $V_L$ or its $\Z_2$-orbifold $\tilde{V}_L$ associated to the $-1$-isometry of the lattice $L$.
Conversely, any lattice VOA associated to an even unimodular lattice of rank $24$ or its $\Z_2$-orbifold has a Virasoro frame whose $\sfr{1}{16}$-code $D$ satisfies (i). In this case, it
was known \cite{DGM} that the VOA structure is determined by the Lie algebra structure of its weight one subspace.

\begin{proposition}\label{PDGM}{\rm (\cite[Table2, Proposition 6.5]{DGM})}
Let $U$ be a holomorphic framed VOA of central charge $24$ with a $\sfr{1}{16}$-code satisfying (i).
%Let $U$ be a holomorphic VOA of central charge $24$ isomorphic to $V_L$ or $\tilde{V}_L$.
Then the isomorphism class of $U$ is uniquely determined by the Lie algebra structure of $U_1$.
In particular, there exist exactly $39$ holomorphic framed VOAs of central charge $24$ with $\sfr{1}{16}$-codes satisfying (i), up to isomorphism.
%isomorphic to $V_L$ or $\tilde{V}_L$.
\end{proposition}

%The proposition above and \cite[Theorem 5.46]{LS} show the following:
%According to \cite{DGM}, we are done.
%the isomorphism type of $V$ is uniquely determined by the Lie algebra structure of $V_1$.
\medskip

\noindent\textbf{Case (ii).}  Suppose that $D$ is a subcode of $\RM(1,4)^{\oplus 3}$. Then $U$ is a simple current extension of
$V^{\otimes3}$, where $V=V_{\sqrt2E_8}^+$.  This case was studied in \cite[Section 5]{LS} and $U\cong \mathfrak{V}(\mathcal{S})$
for some maximal totally singular subspace $\mathcal{S}$ of $R(V)^3$.
%(See Sections \ref{sec:5}   for the definition of $\mathfrak{V}(\mathcal{S})$).
In particular, the following theorem was proved by the uniqueness of simple current extensions
\cite{DM}. (See Sections \ref{sec:2.4},  \ref{sec:5} and \ref{sec:4}  for the definition of $\mathcal{S}(k,m,n,\pm)$,
$\mathfrak{V}(\mathcal{S})$ and $\mathfrak{g}(\Phi)$, respectively.)
% since $V_C$ contains $M_{{\rm RM}(2,4)}^{\otimes3}$ and $M_{{\rm RM}(2,4)}\cong V_{\sqrt2E_8}^+$.

\begin{proposition}\label{PAim}{\rm (\cite[Theorem 5.46]{LS})} Let $U$ be a holomorphic VOA of central charge $24$.
Assume that $U\cong\mathfrak{V}(\mathcal{S})$ for some maximal totally singular subspace $\mathcal{S}$ of $R(V)^3$.
\begin{enumerate}[{\rm (1)}]
\item If $U_1$ is isomorphic to neither $\g(C_{8}F_{4}^2)$ nor $\g(A_{7}C_{3}^2A_{3})$, then the isomorphism class of $U$ is uniquely determined by the Lie algebra structure of $U_1$.
\item If $U_1\cong \g(C_{8}F_{4}^2)$ then $U$ is isomorphic to $\mathfrak{V}(\mathcal{S}(5,3,0,-))$ or $\mathfrak{V}(\mathcal{S}(5,3,2,+))$.
\item If $U_1\cong \g(A_{7}C_{3}^2A_{3})$ then $U$ is isomorphic to $\mathfrak{V}(\mathcal{S}(5,2,1,+))$ or $\mathfrak{V}(\mathcal{S}(5,2,0))$.
\end{enumerate}
\end{proposition}
Hence, it remains to show that $\mathfrak{V}(\mathcal{S}(5,3,0,-))\cong \mathfrak{V}(\mathcal{S}(5,3,2,+))$ and
$\mathfrak{V}(\mathcal{S}(5,2,1,+))\cong \mathfrak{V}(\mathcal{S}(5,2,0))$, which will be achieved in Section \ref{Proof2} (see
Theorems \ref{MT2} and \ref{MT3}).
% and   the isomorphism class of
%$\mathfrak{V}(\mathcal{S})$ is uniquely determined if $U_1\cong \g(C_{8}F_{4}^2)$ or $\g(A_{7}C_{3}^2A_{3})$.
As a consequence, we obtain the following theorem:

\begin{theorem}\label{TCase2} Let $U$ be a holomorphic framed VOA of central charge $24$ with a $\sfr{1}{16}$-code satisfying (ii).
%Assume that $U\cong\mathfrak{V}(\mathcal{S})$ for some maximal totally singular subspace $\mathcal{S}$ of $R(V)^3$.
%Assume that the $\sfr{1}{16}$-code $D$ satisfies (ii).
Then the isomorphism class of $U$ is uniquely determined by the Lie algebra structure of $U_1$.
Excluding the VOAs in Proposition \ref{PDGM}, there exist exactly $10$ holomorphic framed VOAs of central charge $24$ with $\sfr{1}{16}$-codes satisfying (ii), up to isomorphism.
% isomorphic to $\mathfrak{V}(\mathcal{S})$ but not to $V_L$ or $\tilde{V}_L$.
\end{theorem}

%In this case, if $V_1$ is isomorphic to neither $\g(C_{8}F_{4}^2)$ nor $\g(A_{7}C_{3}^2A_{3})$, then we obtain the uniqueness by \cite{LS}.
%the isomorphism type of $U$ is uniquely determined by the Lie algebra structure of $V_1$.
\medskip

\noindent\textbf{Case (iii).}  If $D$ is a subcode of ${\rm RM}(1,4) \oplus \EuD(d_{16}^+)$, then $U$ is a simple current
extension of $V_{\sqrt2E_8}^+\otimes V_{\sqrt2D_{16}^+}^+$ and this case was also studied in \cite[Section 6]{LS}. Moreover, one has the
following proposition by \cite[Theorem 6.17]{LS}, Theorem \ref{TCase2} and the uniqueness of simple current extensions \cite{DM}.

\begin{proposition}\label{PCase3}{\rm (\cite[Theorem 6.17]{LS})} Let $U$ be a holomorphic framed VOA of central charge $24$ with a $\sfr{1}{16}$-code satisfying (iii).
%Assume that $U$ is a simple current extension of $V_{\sqrt2E_8}^+\otimes V_{\sqrt2D_{16}^+}^+$ but is not of
%$(V_{\sqrt2E_8}^+)^{\otimes3}$.
%Assume that the $\sfr{1}{16}$-code $D$ satisfies (iii).
Then the isomorphism class of $U$ is uniquely determined by the Lie algebra structure of $U_1$.
Excluding the VOAs in Proposition \ref{PDGM} and Theorem \ref{TCase2}, there exist exactly $4$ holomorphic framed VOAs of central charge $24$ with $\sfr{1}{16}$-codes satisfying (iii), up to isomorphism.
%that is a simple current extension of $V_{\sqrt2E_8}^+\otimes V_{\sqrt2D_{16}^+}^+$ but is not of
%$(V_{\sqrt2E_8}^+)^{\otimes3}$ and is isomorphic to neither $V_L$ nor $\tilde{V}_L$.
\end{proposition}

Therefore, no extra work is required for this case.

\medskip

\noindent\textbf{Case (iv).}  In \cite{Lam}, holomorphic framed VOAs associated to the subcodes of $D^{ex}$ have been studied
and the Lie algebra structures of their weight one subspaces are determined. It was also shown that the Lie algebra structures of
their weight one subspaces are uniquely determined by the $\sfr{1}{16}$-codes \cite[Theorem 6.78]{Lam}.

Suppose that the $\sfr{1}{16}$-code $D$ satisfies (iv). Then by the classification \cite{BM} (see also
 \textit{\small http://www.st.hirosaki-u.ac.jp/$\sim$betsumi/triply-even/}),  $D$ is
equivalent to $D^{ex}=D_{[10]}$, $D_{[8]}$ or $D_{[7]}$ (see Section \ref{sec:2.7} for the definition of $D_{[k]}$). Moreover, the
Lie algebras of the VOAs associated to $D_{[10]}$, $D_{[8]}$ and $D_{[7]}$  are not included in Cases (i), (ii) and (iii) (see  \cite[Table 1]{LS}).  Therefore, it
remains to show that the VOA structure is uniquely determined by the $\sfr{1}{16}$-code $D$ if $D=D_{[10]}$, $D_{[8]}$ or $D_{[7]}$,
which will be achieved in Corollary \ref{MC}.

\begin{theorem}\label{TCase4} Let $U$ be a holomorphic framed VOA of central charge $24$ with a $\sfr{1}{16}$-code $D$ satisfying (iv).
Then the isomorphism class of $U$ is uniquely determined by the Lie algebra structure of $U_1$.
In particular, there exist exactly $3$ holomorphic framed VOAs of central charge $24$ with $\sfr{1}{16}$-codes satisfying (iv), up to isomorphism.
\end{theorem}

Our main theorem (Theorem \ref{MainT}) will then follows from Propositions \ref{PDGM} and \ref{PCase3} and Theorems \ref{TCase2} and \ref{TCase4}.

\medskip
 The organization of the article is as follows. In Section \ref{pre}, we recall some notions and basic facts  about VOAs and
framed VOAs. In Section \ref{sec:3}, we  study the  framed VOA structures associated to  a fixed $\sfr{1}{16}$-code $D$. We show that
the holomorphic framed VOA structure  is uniquely determined by the $\sfr{1}{16}$-code $D$ if $D$ is a subcode of the exceptional triply even
code $D^{ex}$. In Section \ref{sec:qs},  the isomorphisms between holomorphic VOAs of central charge $24$ associated to some
maximal totally singular subspaces are discussed. We first recall a classification of maximal totally singular subspaces up to certain
equivalence from \cite{LS}. The construction of a VOA $\mathfrak{V}(\mathcal{S})$ from a maximal totally singular subspace
$\mathcal{S}$ is recalled. Some basic properties of the VOA $\mathfrak{V}(\mathcal{S})$ are also reviewed. In Section \ref{sec:4},
the conjugacy classes of certain involutions in lattice VOAs are discussed. The results will then be used in Section \ref{Proof2} to
establish the isomorphisms between some holomorphic VOAs  associated to maximal totally singular subspaces.  In Appendix A, certain
ideals of the weight one subspaces of the VOAs $\mathfrak{V}(\mathcal{S})$ used in Section \ref{Proof2} are described explicitly.

\newpage

\section{Preliminaries}\label{pre}
\begin{center}
{\bf Notations}
\begin{small}
\begin{tabular}{ll}
\ \\
$\langle\ , \ \rangle$& the standard inner product in $\Z_2^n$, $\R^n$ or $(R(V)^3,q_V^3)$.\\
$\mathbf{1}$& the all-one vector in $\Z_2^n$.\\
$\1$& the vacuum vector of a VOA.\\
$\boxtimes$& the fusion product for a VOA.\\
$\langle A\rangle_{\F_2}$& the subspace of $\F_2^n$ spanned by $A$.\\
$\Aut X$& the automorphism group of $X$.\\
$\alpha\cdot\beta$& the coordinatewise product of $\alpha,\beta\in\Z_2^n$.\\
$D\cdot D$& the code ${\rm Span}_{\Z_2}\{\beta\cdot \beta'\mid \beta,\beta'\in D\}$, where $D$ is a binary code.\\
$D^{ex}$& the exceptional triply even code of length $48$.\\
%$\perp $& the orthogonal direct sum of subspaces in a quadratic space.\\
%$A^\perp$& the orthogonal complement of a subspace $A$ in a quadratic space.\\
%$D_{16}^+$& the indecomposable even unimodular lattice of rank $16$,\\
%& whose root lattice is $D_{16}$.\\
%$E_8$& the $E_8$ root lattice, even unimodular lattice of rank $8$.\\
$g\circ M$& the conjugate of a module $M$ for a VOA by an automorphism $g$.\\
$\mathfrak{g}(\Phi)$& the semisimple Lie algebra with the root system $\Phi$.\\
%$g\circ [M]$& the isomorphism class of $g\circ M$.\\
$[M]$& the isomorphism class of a module $M$ for a VOA.\\
$M_C(\alpha,\beta)$& the irreducible module for $V_C$ parametrized by $\alpha\in C^\perp$, $\beta\in\Z_2^n$.\\
$N(\Phi)$& the even unimodular lattice of rank $24$ whose root system is $\Phi$.\\
%$\rho_i$& the $i$-th coordinate projection of a direct sum of quadratic spaces.\\
%$q_V$& the quadratic form on $R(V)$ defined by\\
%& $q_V([M])=0$ and $1$ if $M$ is $\Z$-graded and $(\Z+1/2)$-graded, respectively.\\
$O(R(V),q_V)$& the orthogonal group of the quadratic space $(R(V),q_V)$.\\
$\mathcal{Q}_D$&$\{ \delta:D \to \Z_2^n/ D^\perp\mid \delta \text{ is $\Z_2$-linear and }
(\delta(\be), \mathbf{1}+\be) =0 \text{ for all } \be \in D\}$.\\
%$\Sym_n$& the symmetric group of degree $n$.\\
%$\mathcal{S}$& A maximal totally singular subspace of $R^3$ or $R(V)^3$, where $V=V_{\sqrt2E_8}^+$.\\
$\mathcal{S}(m,k_1,k_2,\varepsilon)$ & the maximal totally singular subspace of $R^3$ defined in theorem \ref{TClassify}.\\
$\mathcal{S}(m,k_1,k_2)$ & the maximal totally singular subspace of $R^3$ defined in theorem \ref{TClassify2}.\\
${\rm supp}(c)$& the support of $c=(c_i)\in\Z_2^n$, that is, the set $\{i\mid c_i\neq0\}$.\\
$\Sym_n$& the symmetric group of degree $n$.\\
%$\mathcal{S}(R)$& the set of singular vectors in $R$\\
%$\mathcal{S}(R)^\times$& the set of non-zero singular vectors in $R$\\
%$\overline{\mathcal{S}(R)}$& the set of non-singular vectors in $R$\\
%${\rm Stab}_G(A)$& the setwise stabilizer of $A$ in a group $G$.\\
%${\rm Stab^{pt}}_G(A)$& the pointwise stabilizer of $A$ in a group $G$.\\
%$(R,q)$& A non-singular quadratic space $R$ of plus type with quadratic form $q$ over $\F_2$.\\
%$(R^k,q^k)$& the orthogonal direct sum of $k$ copies of $(R,q)$.\\
$R(U)$& the set of all isomorphism classes of irreducible modules for a VOA $U$.\\
$(R(V),q_V)$& the $10$-dimensional quadratic space $R(V)$ associated to $V=V_{\sqrt2E_8}^+$.\\
$L(\Phi)$& the root lattice with root system $\Phi$.\\
%$V$&
%A simple rational $C_2$-cofinite self-dual VOA of CFT type, or
%$V$& the VOA $V_{\sqrt2E_8}^+$.\\
$V_C$& the code VOA associated to binary code $C$.\\
$V_L$& the lattice VOA associated with even lattice $L$.\\
$V_L^+$& the fixed point subVOA of $V_L$ with respect to a lift of the $-1$-isometry of $L$.\\
$\tilde{V}_L$& the $\Z_2$-orbifold of $V_L$ associated to the $-1$-isometry of $L$.\\
%$V^3$& the tensor product of $3$ copies of a VOA $V$.\\
$\mathfrak{V}(\mathcal{S})$& the holomorphic VOA associated to a maximal totally singular subspace $\mathcal{S}$.\\
%$\mathfrak{V}(\mathcal{T})$& the module associated to a subset $\mathcal{T}$ of $R(V)^k$.
% or %$R(X)\oplus R(V)$.\\
%$w^k$& the map from $R(V)^k$ to $\{0,1,\dots,2k\}$ (see Section 4.4).\\
%$X$& $X=V_{\sqrt2D_{16}^+}^+$.\\
\end{tabular}
\end{small}
\end{center}

\medskip

\subsection{Vertex operator algebras}
Throughout this article, all VOAs are defined over the field $\C$ of complex numbers. We recall the notion of vertex operator
algebras (VOAs) and modules from \cite{Bo,FLM,FHL}.

A {\it vertex operator algebra} (VOA) $(V,Y,\1,\omega)$ is a $\Z_{\ge0}$-graded
 vector space $V=\bigoplus_{m\in\Z_{\ge0}}V_m$ equipped with a linear map

$$Y(a,z)=\sum_{i\in\Z}a_{(i)}z^{-i-1}\in ({\rm End}(V))[[z,z^{-1}]],\quad a\in V$$
and the {\it vacuum vector} $\1$ and the {\it conformal element} $\omega$
satisfying a number of conditions (\cite{Bo,FLM}). We often denote it by $V$ or
$(V,Y)$.

Two VOAs $(V,Y,\1,\omega)$ and $(V^\prime,Y',\1',\omega')$ are said to be {\it
isomorphic} if there exists a linear isomorphism $g$ from $V$ to $V^\prime$
such that $$ g\omega=\omega'\quad {\rm and}\quad gY(v,z)=Y'(gv,z)g\quad
\text{ for all } v\in V.$$ When $V=V'$,  such a linear isomorphism is called an
{\it automorphism}. The group of all automorphisms of $V$ is called the {\it
automorphism group} of $V$ and is denoted by $\Aut V$.

A {\it vertex operator subalgebra} (or a {\it subVOA}) is a graded subspace of
$V$ which has a structure of a VOA such that the operations and its grading
agree with the restriction of those of $V$ and that they share the vacuum vector.
When they also share the conformal element, we will call it a {\it full subVOA}.

An (ordinary) module $(M,Y_M)$ for a VOA $V$ is a $\C$-graded vector space $M=\bigoplus_{m\in\C} M_{m}$ equipped with a linear map
$$Y_M(a,z)=\sum_{i\in\Z}a_{(i)}z^{-i-1}\in ({\rm End}(M))[[z,z^{-1}]],\quad a\in V$$
satisfying a number of conditions (\cite{FHL}).
We often denote it by $M$ and its isomorphism class by $[M]$.
The {\it weight} of a homogeneous vector $v\in M_k$ is $k$.
A VOA is said to be  {\it rational} if any module is completely reducible.
A rational VOA is said to be {\it holomorphic} if itself is the only irreducible module up
to isomorphism.
A VOA is said to be {\it of CFT type} if $V_0=\C\1$, and is said to be {\it $C_2$-cofinite} if $\dim V/{\rm Span}_\C\{u_{(-2)}v\mid u,v\in V\}<\infty$.

Let $M$ be a module for a VOA $V$ and let $g$ be an automorphism of $V$.
Then the module $g\circ M$ is defined by $(M,Y_{g\circ M})$, where $Y_{g\circ M}(v,z)=Y_M(g^{-1}(v),z)$, $v\in V$.
Note that if $M$ is irreducible then so is $g\circ M$.

Let $V$ be a VOA of CFT type. Then the $0$-th product gives a Lie algebra structure on $V_1$. Moreover, the operators
$v_{(n)}$, $v\in V_1$, $n\in\Z$, define  a representation of the affine Lie algebra associated to $V_1$. Note that $\Aut V$ acts
on the Lie algebra $V_1$ as an automorphism group.

\subsection{Fusion products and simple current extensions}

Let $V^0$ be a simple rational $C_2$-cofinite VOA of CFT type and let $W^1$ and $W^2$ be $V$-modules.
It was shown in \cite{HL} that the $V^0$-module $W^1\boxtimes_{V^0} W^2$, called the {\it fusion product}, exists.
A $V^0$-module $M$ is called a {\it simple current} if for any irreducible $V^0$-module $X$, the fusion product $M\boxtimes_{V^0} X$ is also irreducible.

%\begin{lemma} Let $g$ be an automorphism of $V$.
%Then $(g\circ W^1)\boxtimes_V(g\circ W^2)\cong g\circ (W^1\boxtimes_V W^2)$ as $V$-modules.
%\end{lemma}

Let $\{V^\alpha\mid \alpha\in D\}$ be a set of inequivalent irreducible $V^0$-modules indexed by an abelian group $D$.
A simple VOA $V_D=\bigoplus_{\alpha\in D}V^\alpha$ is called a {\it simple current extension} of $V^0$ if it carries a $D$-grading and every $V^\alpha$ is a simple current.
Note that $V^\alpha\boxtimes_{V^0}V^\beta\cong V^{\alpha+\beta}$.

\begin{proposition}{\rm (\cite[Proposition 5.3]{DM})}\label{PSCE} Let $V^0$ be a simple rational $C_2$-cofinite VOA of CFT type and let $V_D=\bigoplus_{\alpha\in D}V^\alpha$ and $\tilde{V}_D=\bigoplus_{\alpha\in D}\tilde{V}^\alpha$ be simple current extensions of $V^0$.
If $V^\alpha\cong \tilde{V}^\alpha$ as $V^0$-modules for all $\alpha\in D$, then $V_D$ and $\tilde{V}_D$ are isomorphic VOAs.
\end{proposition}

\subsection{Lattice VOAs and $\Z_2$-orbifolds}
Let $L$ be an even unimodular lattice and let $V_L$ be the lattice VOA associated with $L$ (\cite{Bo,FLM}).
Then $V_L$ is holomorphic (\cite{Do}).
Let $\theta\in \aut V_L$ be a lift of $-1\in\aut L$ and let $V_L^+$ denote the subVOA of $V_L$ consisting of vectors in $V_L$ fixed by $\theta$.
Let $V_L^T$ be a unique irreducible $\theta$-twisted module for $V_L$ and let $V_L^{T,+}$ be the irreducible $V_L^+$-submodule of $V_L^T$ with integral weights. %Then one can show (cf. \cite{FLM,DGH,LY}) that
Set $$ \tilde{V}_L = V_L^+ \oplus V_L^{T,+}.$$ Then $ \tilde{V}_L$
has a unique holomorphic VOA structure by extending its
$V_L^+$-module structure, up to isomorphism (\cite{FLM,DGM}).
The VOA  $\tilde{V}_L$ is often called the {\it $\Z_2$-orbifold} of $V_L$.
More generally, for an involution $g$ in $\Aut V_L$, we can consider the same procedure.
If we obtain a VOA as a simple current extension of the subVOA $V_L^g$ of $V_L$ fixed by $g$, we call it the $\Z_2$-orbifold of $V_L$ associated to $g$.
%\begin{remark}\label{Rem2}
%If two involutions $g$ and $g'$ in $\Aut V_L$ are conjugate, then the $\Z_2$-orbifolds of $V_L$ associated to $g$ and $g'$ are isomorphic.
%\end{remark}

\subsection{Code VOAs and framed VOAs}

In this subsection, we review the notion of code VOAs and framed VOAs from \cite{M1,M2,DGH,M3}.

Let $\vir=\bigoplus_{n\in\Z}\C L_n \oplus \C \mathbf{c}$ be the Virasoro algebra.
%That means  the $L_n$ satisfy the commutator relations:
%\[
%\begin{split}
% [L_m,L_n]=(m-n)L_{m+n}+ \frac{1}{12}(m^3-m) \delta_{m+n,0} \mathbf{c},\quad\text{ and }\quad
% [L_m,\mathbf{c}]=0.
%\end{split}
%\]
For any $c,h \in \C$, we denote by $L(c,h)$ the irreducible highest weight
module of $\vir$ with central charge $c$ and highest weight $h$.
It was shown in \cite{FZ} that $L(c,0)$ has a natural VOA structure. We call it the
\emph{simple Virasoro VOA} with central charge $c$.

\begin{definition}\label{df:3.1}
Let $V=\bigoplus_{n=0}^\infty V_n$ be a VOA. An element $e\in V_2$ is  called
an {\it Ising vector}
% or  a  {\it Virasoro vector of central charge $1/2$}
if the
subalgebra $\vir(e)$ generated by $e$ is isomorphic to $L( \shf,0)$ and $e$ is
the conformal element of $\vir(e)$. Two Ising vectors $u,v\in V$ are said to be
{\it orthogonal} if $[Y(u,z_1), Y(v,z_2)]=0$.
\end{definition}

\begin{remark}
It is well-known that $L(\shf,0)$ is rational
% i.e., all  $L(\shf, 0)$-modules are completely reducible,
and has only three inequivalent irreducible modules
$L(\shf,0)$, $L(\shf,\shf)$ and $L(\shf,\sfr{1}{16})$. The fusion products of
$L(\shf,0)$-modules are computed in \cite{DMZ}:
\begin{equation}\label{eq:3.1}
\begin{array}{l}
  L(\shf,\shf)\fusion L(\shf,\shf)=L(\shf,0),
  \q
  L(\shf,\shf)\fusion L(\shf,\sfr{1}{16})=L(\shf,\sfr{1}{16}),
  \vsb\\
  L(\shf,\sfr{1}{16})\fusion L(\shf,\sfr{1}{16})=L(\shf,0)\oplus L(\shf,\shf).
\end{array}
\end{equation}
\end{remark}

\begin{definition}\label{df:3.2}
  (\cite{DGH})
A simple VOA $V$ is said to be \emph{framed} if there exists a set $\{e^1,
\dots,e^n\}$ of mutually orthogonal Ising vectors of $V$ such that their sum
$e^1+\cds +e^n$ is equal to the conformal element of $V$.   The subVOA  $T_n$
generated by $e^1,\dots,e^n$ is thus isomorphic to
  $L(\shf,0)^{\otimes n}$ and  is called
  a \emph{Virasoro frame} of $V$.
\end{definition}

\begin{theorem}{\rm (\cite{DGH})} Any framed VOA is rational, $C_2$-cofinite, and of CFT type.
\end{theorem}

Given a framed VOA $V$ with a Virasoro frame $T_n$, one can associate two binary codes
$C$ and $D$ of length $n$ to $V$ and $T_{n}$ as follows:
Since $T_n= L(\shf,0)^{\otimes n}$ is rational, $V$ is a completely reducible
$T_n$-module. That is,
\begin{equation*}\label{2.1}
V \cong \bigoplus_{h_i\in\{0,\sfr{1}{2},\sfr{1}{16}\}}
m_{h_1,\ldots, h_{n}}L(\shf,h_1)\otimes\cdots\otimes L(\shf,h_n),
\end{equation*}
where the nonnegative integer $m_{h_1,\ldots,h_n}$ is the multiplicity of
$L(\shf,h_1)\otimes\cdots\otimes L(\shf, h_n)$ in $V$.
It was shown in \cite{DMZ} that all the multiplicities are finite and that $m_{h_1,\ldots,h_n}$ is at most $1$ if all $h_i$ are different from $\sfr{1}{16}$.

%\begin{definition}\label{tauword}
Let $U\cong L(\shf,h_1)\otimes\cdots\otimes L(\shf,h_n) $ be an irreducible
module for $T_{n}$. Let $\tau(U)$ denote the binary word $\beta=(\beta_1, \dots,\beta_n)\in
\Z_2^n$ such that
\begin{equation}%\label{eq:1.1}
\beta_i=
\begin{cases}
0 & \text{ if } h_i=0 \text{ or } \shf,\\
1 & \text{ if } h_i=\sfr{1}{16}.
\end{cases}
\end{equation}
%\end{definition}
%We call it $\sfr{1}{16}$-word (or $\tau$-word) of $U$.

For any $\beta\in \Z_2^n$, denote by $V^\beta$ the sum of all irreducible
submodules $U$ of $V$ such that $\tau(U)=\beta$.
%\begin{definition}
Set $ D:=\{\beta\in \Z_2^n \mid V^\beta \ne 0\}$. Then $D$ becomes a
binary code of length $n$.
We call $D$ the {\it $\sfr{1}{16}$-code} with respect to $T_n$.
Note that $V$ can be written as a sum
$$V=\bigoplus_{\beta\in D} V^\beta.$$
%\end{definition}
%\medskip
For any $\alpha=(\alpha_1,\ldots,\alpha_n)\in \Z_2^n$, let
$M_\alpha$ denote the $T_n$-submodule $m_{h_1,\ldots,h_n}L(\shf,h_1)\otimes\cdots\otimes L(\shf,h_n) $ of $V$, where
$h_i=\shf$ if $\alpha_i=1$ and $h_i=0$ elsewhere.  Note that $m_{h_1,\ldots,h_n}\leq
1$ since $h_i\neq \sfr{1}{16}$.
Set %\begin{definition}
$C:=\{\alpha\in \Z_2^n \mid M_\alpha \neq 0\}$. Then $C$ also forms a binary code
%We call $C$ the {\it $\sfr{1}{2}$-code} with respect to $T_n$.
and $V^0=\bigoplus_{\alpha\in C}M_\alpha$.
%\end{definition}
The code VOA $V_C$ associated to a binary code $C$ was defined in \cite{M1}.

\begin{definition}\label{CodeVOA} (\cite{M1})
A framed VOA $V$ is called a {\it code VOA} if $D=0$, equivalently, $V=V^0$.
\end{definition}

\begin{proposition} {\rm (\cite[Theorem 4.3]{M1}, \cite[Theorem 4.5]{M2}, \cite[Proposition 2.16]{DGH})} For any even code $C$, there exists the unique code VOA isomorphic to $\bigoplus_{\alpha\in C}M_\alpha$, up to isomorphism.
\end{proposition}

%\begin{remark}
%The VOA $V^0$  is often called the {\it code VOA} associated to $C$  and is denoted
%by $V_C$ (\cite{M1}).
%For a coset $\delta+C$ of $C$ in $\Z_2^n$, we also
%denote
%$$M_{\delta+C}= \bigoplus_{c\in\delta+ C} M_c.$$ In this case, $M_{\delta+C}$
%is an irreducible $V_C$-module (cf.\ \cite{M2}).
%\end{remark}
Summarizing, there exists a pair of binary codes $(C,D)$ such that
\[
V=\bigoplus_{\be\in D} V^\be \quad \text{ and } \quad V^0=\bigoplus_{\alpha\in C}M_\alpha.
\]
%The codes $(C,D)$ are called the {\it structure codes} of a framed VOA $V$
%associated to the frame $T_{n}$.
%We also call the code $D$ the $\frac{1}{16}$-code and the code $C$ the
%$\frac{1}2$-code of $V$ with respect to $T_n$.
Note that all $V^\be$, $\be\in D$ , are irreducible $V^0$-modules.

%\medskip

Since $V$ is a VOA, its weights are integers and we have the lemma.
\begin{lemma}\label{Lwt}
\begin{enumerate}[{\rm (1)}]
\item The code $D$ is triply even, i.e., $\wt (\beta)\equiv 0\mod 8$ for all $\beta \in D$.

\item The code $C$ is even.
\end{enumerate}
\end{lemma}

The following theorems are well-known.

\begin{theorem} {\rm (\cite[Theorem 2.9]{DGH} and
\cite[Theorem 6.1]{M3})}
Let $V$ be a framed VOA with binary codes $(C,D)$. Then, $V$ is holomorphic
if and only if $C=D^\perp$.
\end{theorem}

\begin{theorem} {\rm (\cite[Theorem7]{LY})}\label{PFrameSCE} Let $V=\bigoplus_{\beta\in D}V^\beta$ be a framed VOA.
Then $V$ is a $D$-graded simple current extension of $V^0$.
\end{theorem}

%In \cite{LY}, the structure of a general framed VOA has been studied in detail%In particular, the following is established (see \cite[Theorem 10]{LY}).

%\begin{theorem}\label{tecode}
%Let $D$ be a linear binary code of length $16k, k\in \Z_{>0}$.  Then $D$ can be
%realized as the $\frac{1}{16}$-code of a holomorphic framed VOA of central
%charge $8k$ if and only if (1) $D$ is triply even and (2)  the all-one vector
%$(1^{16k})\in D$.
%\end{theorem}
%\newpage

\subsection{Representation theory of code VOAs}
In this subsection, we review representation theory of code VOAs from \cite{M2,M3,DGL,LY}.

Let $C$ be an even binary code of length $n$ and $V_C$ the code VOA associated to $C$.
Let us recall a parametrization of irreducible $V_C$-modules by codewords from \cite[Section 4.2]{LY}.
Let $\beta\in C^\perp$ and $\gamma\in \Z_2^n$.
We define a weight vector
$h_{\beta,\gamma}=(h^1_{\beta,\gamma},\dots,h^n_{\beta,\gamma})$,
$h_{\beta,\gamma}^i \in \{ 0,\sfr{1}{2},\sfr{1}{16}\}$ by
$$
  h^i_{\beta,\gamma}:=
  \begin{cases}
  \dfrac{1}{16} & \text{  if } \beta_i=1,
  \vsb\\
  \dfrac{\gamma_i}2 & \text{ if } \beta_i=0.
  \end{cases}
$$
Let
$$
  L(h_{\beta,\gamma}):= L(\shf,h^1_{\beta,\gamma})\tensor \cds
  \tensor L(\shf,h^n_{\beta,\gamma})
$$
be the irreducible $L(\shf,0)^{\tensor n}$-module with the weight
$h_{\beta,\gamma}$.
Let $H$ be a maximal self-orthogonal subcode of $C_\beta=\{\alpha\in C\mid {\rm supp}(\alpha)\subset {\rm supp}(\beta)\}$.
Then there exists an irreducible character $\tilde{\chi}_\gamma$ of the central extension of $H$ such that $L(h_{\beta,\gamma})\otimes\tilde{\chi}_\gamma$ is an irreducible $V_H$-module.
Moreover, we obtain an irreducible $V_C$-module $M_C(\beta,\gamma)$ as its induced module.
%$$M_C(\beta,\gamma)=\ind_{V_H}^{V_C}(L(h_{\beta,\gamma}),\tilde{\chi}_{\gamma})$$ is an irreducible $V_C$-module.

\begin{theorem}\label{thm:4.3} {\rm (\cite[Theorem 5.3]{M2})} Every irreducible $V_C$-module is isomorphic to an induced module $M_C(\beta,\gamma)$ and its module structure is
  uniquely determined by the structure of a $V_H$-submodule.
\end{theorem}

%\begin{remark}\label{lem:4.5}  (cf. Lemma 4.5 of \cite{LY})
 % The module structure of $M_C(\beta,\gamma)$ is independent of the choice of the maximal self-orthogonal subcode $H$ of $C_\beta$ and the choice of the section from the
 % radical of $C_\beta$ to $H$.
%\end{remark}

Next let us review some basic properties of $M_C(\beta,\gamma)$.
%The details can be found in \cite{LY}.

\begin{lemma}\label{lem:4.6}  {\rm (\cite[Lemma 5.8]{DGL} and \cite[Lemma 3]{LY})}
  Let $\beta,\beta'\in C^\perp$ and $\gamma,\gamma'\in \Z_2^n$.
  Then the irreducible $V_C$-modules $M_C(\beta,\gamma)$ and
  $M_C(\beta',\gamma')$ are isomorphic if and only if
  $$
    \beta=\beta' \q \mbox{and} \q \gamma+\gamma' \in C+ H^{\perp_\beta},
  $$
  where $H^{\perp_\beta}=\{ \alpha \in \Z_2^n \mid \supp(\alpha)\subset
  \supp(\beta) \mbox{ and } \la \alpha,\delta\ra =0 \mbox{ for all }
  \delta \in H\}$.
\end{lemma}

\begin{remark}\label{rem:4.6} {\rm (\cite[Remark 6]{LY})} If $C$ is even, $n\equiv0\pmod{16}$, and $C^\perp$ is triply even, then $H^{\perp_\beta}\subset C$ in Lemma \ref{lem:4.6}.
\end{remark}

\begin{lemma}\label{lem:4.12} {\rm (\cite[Lemma 7]{LY})}
  Let $\alpha,\beta,\gamma \in \Z_2^n$ with $\beta\in C^\perp$.
  Then
  $$
    M_C(0,\alpha)\fusion_{V_C}M_C(\beta,\gamma)
    \cong M_C(\beta,\alpha+\gamma).
  $$
  Moreover, the difference between the top weight of $M_C(\beta,\gamma)$
  and that of $M_C(\beta,\alpha+\gamma)$ is congruent to
  $\la \al, \al+\be\ra/2 $ modulo $\Z$.
\end{lemma}

\begin{definition}
Let $C$ be an even code and $\alpha\in \Z_2^n$. Define the map $\sigma_\alpha: V_C \to V_C$ by
\[
\sigma_\alpha(u) = (-1)^{\la \alpha, \beta\ra} u \quad \text{ for } u \in M_\beta,\ \beta\in C.
\]
It is known \cite{M1} that $\sigma_\alpha$ is an automorphism of $V_C$.
\end{definition}

Next lemma plays an important role in Section 3.

\begin{lemma}\label{Lsigma} Let $C$ be an even code of length $n$ and let $\beta\in C^\perp$, $\alpha,\gamma\in\Z_2^n$.
Then
\[
\sigma_\al\circ M_C(\beta,\gamma)\cong M_C(0,\alpha\cdot\beta)\fusion_{V_C}M_C(\beta,\gamma),
\]
where $\alpha=(\alpha_i),\beta=(\beta_i)\in\F_2^n$, $\alpha\cdot\beta=(\alpha_i\beta_i)\in\F_2^n$.
\end{lemma}

\begin{proof}
Let $e_i$ be the vector in $\F_2^n$ which is $1$ in the $i$-th coordinate and $0$ in the other coordinates.
Then $\sigma_\al =\prod_{i\in \supp(\al)} \sigma_{e_i}$.
By Lemma \ref{lem:4.12}, it suffices to show that
\[
\sigma_{e_i} \circ M_C(\beta,\gamma) \cong
\begin{cases}
M_C(0,e_i)\fusion_{V_C}M_C(\beta,\gamma) & \text{ if } i\in {\rm supp}(\be), \\
M_C(\beta,\gamma) & \text{ if } i\notin {\rm supp}(\be).
\end{cases}
\]

Let $H$ be a maximal self-orthogonal subcode of $C_\be$.
Then by Theorem \ref{thm:4.3}, the $V_C$-module structure is uniquely determined by a $V_H$-submodule structure.

If $i\notin {\rm supp}(\be)$, then $\sigma_{e_i}$ is trivial on $V_H$. Hence
$\sigma_{e_i} \circ\left( L(h_{\be, \gamma})\otimes \tilde{\chi}_{\gamma}\right) \cong
L(h_{\be, \gamma})\otimes \tilde{\chi}_{\gamma}$ as $V_H$-modules, and we have
$\sigma_{e_i} \circ M_C(\beta,\gamma)\cong M_C(\beta,\gamma)$ as $V_C$-modules.

Assume $i\in{\rm supp}(\be)$.
Then $h_{\beta,\gamma}=h_{\beta,\gamma+e_i}$.
Let $c=(c_i)\in H$. Then $\sigma_{e_i}$ acts on the submodule $\otimes_{i=1}^n L(\sfr{1}{2},c_i/2)$ of $V_H$ by the scaler $(-1)^{\la e_i, c\ra}$. Therefore,
\[
\sigma_{e_i} \circ\left( L(h_{\be, \gamma})\otimes \tilde{\chi}_{\gamma}\right) \cong
L(h_{\be, \gamma})\otimes \chi,
\]
where $\chi(c) =(-1)^{\la e_i, c\ra} \tilde{\chi}_\gamma(c)  = \tilde{\chi}_{\gamma+e_i}(c)$ for all $c\in H$, which proves $\sigma_{e_i} \circ M_C(\beta,\gamma)\cong M_C(\beta,e_i+\gamma)$.
The desired result follows from $M_C(\beta,e_i+\gamma)\cong M_C(0,e_i)\fusion_{V_C}M_C(\beta,\gamma)$ (Lemma \ref{lem:4.12}).
\end{proof}

\section{Uniqueness of framed VOAs associated to subcodes of $D^{ex}$} \label{sec:3}
In this section, we will  show that the isomorphism class of a framed VOA is uniquely determined by the $\sfr{1}{16}$-code $D$ if
$D$ is a subcode of the $9$-dimensional exceptional triply even code $D^{ex}$ of length $48$.

\subsection{Exceptional triply even code of length $48$}\label{sec:2.7}

First we recall the properties of the $9$-dimensional exceptional triply even code $D^{ex}$ of length $48$ given by \cite{BM}. (see also \cite{Lam}.)

Let $X=\{1,2, \dots, 10\}$ be a set of $10$ elements and let
\[
\Omega:= \binom{X}{2} =\big\{\{i,j\}\mid \{i,j\}\subset X \big\}
\]
be the set of all $2$-element subsets of $X$. Then $|\Omega|=\binom{10}{2} =45$. The triangular graph on $X$ is a graph whose
vertex set is $\Omega$ and two vertices $S, S'\in \Omega$ are joined by an edge if and only if $|S\cap S'|=1$. We will denote by
$\mathcal{T}_{10}$ the binary code generated by the row vectors of the incidence matrix of the triangular graph on $X$.
Note that $\dim\mathcal{T}_{10}=8$.

%\begin{remark}
%Note that the entries of an incidence matrix are either $0$ or $1$ and we shall
%view $0$ and $1$ as integers modulo $2$.
%\end{remark}

\begin{notation}
For $\{i,j\}\in \Omega$, let $\gamma_{\{i,j\}}$ be the binary word supported at
$\{\{k,\ell\}\mid |\{i,j\}\cap \{k,\ell\}|=1\}$, i.e., the set of all vertices
joining to $\{i,j\}$.  Note that
\begin{equation}\label{supp}
{\rm supp}(\gamma_{\{i,j\}})= \{\{i,k\}\mid k\in X\setminus \{i,j\}\}
\cup \{\{j,k\}\mid k\in X\setminus \{i,j\}\}
\end{equation}
and $\wt(\gamma_{\{i,j\}})=16$. For convenience, we often identify
$\gamma_{\{i,j\}}$ with its support.
\end{notation}

Now let $\iota: \Z_2^{45} \to \Z_2^{48}$ be the map defined by $\iota(\alpha)=(\alpha,
0,0,0)$. Then we can embed $\mathcal{T}_{10}$ into $\Z_2^{48}$ using $\iota$.

\begin{definition}\label{df:Dex}
Denote by $D^{ex}$ the binary code generated by $\iota(\mathcal{T}_{10}) $ and
the all-one vector $\mathbf{1}$ in $\Z_2^{48}$. Clearly, $\dim D^{ex}=9$.
\end{definition}

%\begin{notation}
%Let $D^{ex}$ be the triply even code defined in Definition \ref{df:Dex} and let
%$C^{ex}=(D^{ex})^\perp$ be the dual code. Then $\dim C^{ex}=39$. Again, we often identify a codeword with its support.
%\end{notation}

\begin{notation}
For $\alpha=(\alpha_1, \cdots,\alpha_n), \beta=(\be_1, \cdots,\be_n)\in \Z_2^n$, we will denote by $\al\cdot \be$ the
coordinatewise product of $\al$ and $\beta$, i.e., $\al\cdot \be= (\al_1\be_1, \dots, \al_n \be_n)$.
For a binary code $D$, we also denote the code ${\rm Span}_{\Z_2}\{\beta\cdot\beta'\mid \beta,\beta'\in D\}$ by $D\cdot D$.
%We will also use $p_{\beta}$
%to denote the natural projection of $\Z_2^{48}$ to the support of $\beta$, where $\beta\in \Z_2^{48}$.
\end{notation}

\begin{lemma} {\rm (\cite[Lemma 16]{BM})}\label{Cex}
 For any $2\leq i<j\leq 10$, denote $\be_{i,j} =\gamma_{\{1,i\}}\cdot \gamma_{\{1,j\}}$. Then the set $\{ \iota(\be_{i,j})\mid 2\leq i<j\leq 10\}\cup \{\mathbf{1}\} \cup \{ (0^{45},1,1,0),(0^{45},1,0,1)\}$ is a basis of  $(D^{ex})^\perp$.
\end{lemma}

\begin{proposition}\label{LID}
Let $D$ be a $d$-dimensional subcode of $D^{ex}$ containing $\mathbf{1}$ and let $B=\{\mathbf{1}, \be_1, \cdots, \be_{d-1}\}$
be a basis of $D$. Then the set $\mathcal{B}=\{\mathbf{1}\} \cup \{\be\cdot \be'\mid \be, \be'\in B, \be\neq \be'\}$ is linearly
independent. In particular, $\dim (D\cdot D)=\binom{d}{2}+1$.
\end{proposition}

\begin{proof} It suffices to consider the case where $D=D^{ex}$.
Note that $d=9$.
%Set $\mathcal{P}={\rm Span}_{\Z_2}\{ \be\cdot\be'|\be, \be'\in D\}$.
Since $|\mathcal{B}|\le \binom{9}{2}+1=37$ and $\langle\mathcal{B}\rangle_{\Z_2}=D\cdot D$, we have $\dim (D\cdot D)\leq 37$. Therefore, it suffices to show that $\dim(D\cdot D)\geq 37$.
%Since $D$ is triply even, $\mathcal{P}\subset D^\perp$.
%Hence $\dim \mathcal{P}\le 48-9=37$.

Let $\be_{i,j} =\gamma_{\{1,i\}}\cdot \gamma_{\{1,j\}}$ be defined as in Lemma \ref{Cex}. Then
 $\iota(\be_{i,j})\in D\cdot D$ for all $i,j$.
By Lemma \ref{Cex}, $\{\mathbf{1}\}\cup \{\iota(\be_{i,j})\mid 2\leq i <j\leq 10\}$ is a linearly independent subset of $D\cdot D$ with $37$ vectors.
%Since $B$ is a basis of $D$, $\mathcal{B}$ generates $\mathcal{P}$
Hence, $\dim(D\cdot D)\ge37$, and thus $\dim (D\cdot D)=37$ as desired.
 \end{proof}

 Next we recall a notation for denoting subcodes of $D^{ex}$ from \cite{Lam}.
%, which will be used to study the subcodes of $D^{ex}$.

\begin{notation}
Let $\lambda =(\lambda_1, \dots, \lambda_m)$ be a partition of $10$. Let $X_1,
\dots, X_m$ be subsets of $X$ such that $X=\bigcup_{i=1}^m X_i$ and
$|X_i\cap X_j|=\lambda_i\delta_{i,j}$ for $1\leq i,j\leq m$.
Let $D_{[\lambda_1, \dots,\lambda_m]}$ denote the code of length $48$ generated by the all-one vector
$\mathbf{1}$ and $\{ \gamma_{\{i,j\}} \mid \{i,j\}\subset X_k, 1\leq k\leq m\}$.
%Set $W_{\lambda} = \{ \gamma_{\{i,j\}} \mid \{i,j\}\subset X_k, 1\leq k\leq m\}$. Then $W_{\lambda}$ and the all-one vector
%$\mathbf{1}\in \Z_2^{48}$ generates a subcode of $D^{ex}$. We will denote this code by $D_{[\lambda_1, \dots,\lambda_m]}$.
For convenience, we often omit the $1$'s in the partition. For example, $D_{[8]}=D_{[8,1,1]}$  and $D_{[7]}=D_{[7,1,1,1]}$. Note
also that $D^{ex}= D_{[10]}$.
\end{notation}

\begin{remark}
It is clear by the definition that the code $D_{[\lambda_1, \dots, \lambda_m]}$ is
uniquely determined by the shape of the partition $(\lambda_1, \dots,\lambda_m)
$ up to the action of $\Sym_{10}$.
\end{remark}

\subsection{Framed VOA structures associated with a certain $\sfr{1}{16}$-code}
In this subsection, we show that the holomorphic framed VOA structure is uniquely determined by the  $\sfr{1}{16}$-code under
certain assumptions on the $\sfr{1}{16}$-code. As a corollary, we prove that the framed VOA structure is uniquely determined if
the $\sfr{1}{16}$-code is a subcode of the exceptional triply even code of length $48$.

\begin{definition}\label{S}
Let $D$ be a triply even code of length $n$ divisible by $16$.
Define
\[
\mathcal{Q}_D= \{ \delta:D \to \Z_2^n/ D^\perp\mid \delta \text{ is $\Z_2$-linear and }
\langle\delta(\be), \mathbf{1}+\be\rangle =0 \text{ for all } \be \in D\} .
\]
Note that $\mathcal{Q}_D$ is a linear subspace of ${\rm Hom}_{\Z_2}(D, \Z_2^n/D^\perp)$.
\end{definition}

\begin{lemma}\label{dS}
Let $D$ be a $d$-dimensional triply even code of length $n$ divisible by $16$. Assume that $D$ contains the all-one vector
$\mathbf{1}$.
\begin{enumerate}[{\rm (1)}]
\item Let $B=\{\mathbf{1}, \be_1, \dots, \be_{d-1}\}$ be a basis of $D$ and let $\delta\in{\rm Hom}_{\Z_2}(D,\Z_2^n/D^\perp)$.
Then $\delta\in\mathcal{Q}_D$ if and only if both {\rm (a)} $\langle\delta({\be}), \mathbf{1}+\be\rangle=0$ and {\rm (b)}
$\langle\delta({\be}), \be'\rangle=\langle\delta({\be'}), \be\rangle$ hold for all $\be, \be'\in B$.
\item $\dim \mathcal{Q}_D = 1+\binom{d}{2}$.
\end{enumerate}
\end{lemma}

\begin{proof} (1): Assume $\delta\in\mathcal{Q}_D$.
By the definition of $\mathcal{Q}_D$, (a) holds. Moreover, by the definition of $\mathcal{Q}_D$ and the $\Z_2$-linearity of
$\delta$, we have
$\langle\delta({\beta+\beta'}),\mathbf{1}+\beta+\beta'\rangle=\langle\delta(\beta),\beta'\rangle+\langle\delta(\beta'),\beta\rangle=0$
for all $\beta,\beta'\in B$. Hence (b) holds.

Conversely, we assume (a) and (b). Then $\delta\in\mathcal{Q}_D$ since for $\sum_{\beta\in B} c_\beta\beta\in D$,
$$\langle\delta(\sum_{\beta\in B} c_\beta\beta),\mathbf{1}+\sum_{\beta\in B}c_\beta\beta\rangle=\sum_{\beta,\beta'\in
B\atop{\beta\neq\beta'}}c_\beta c_{\beta'}\langle\delta(\beta),\beta'\rangle+\sum_{\beta\in
B}c_\beta\langle\delta(\beta),\mathbf{1}+\beta\rangle=0.$$

(2): By (1), in order to determine $\dim \mathcal{Q}_D$, it suffices to count the possibilities of the images of elements in $B$ satisfying (a) and (b).
Note that for $\beta=\mathbf{1}$, (a) is automatically satisfied. For $\beta\neq \mathbf{1}$,
the subspace of $\Z_2^n /D^\perp$ that satisfies (a) has dimension $d-1$.
Therefore, to obtain a subset  $\{\delta(\mathbf{1}), \delta({\be_1}), \dots, \delta( \be_{d-1})\}$ satisfying (a) and (b),
we have  $2^d$ choices for $\delta(\mathbf{1})$ and  $2^{(d-1)-1}$ choices for $\delta({\be_1})$ that satisfies (a) and   $\langle\delta({\be_1}), \mathbf{1}\rangle=\langle\delta({\mathbf{1}}), \be_1\rangle$. Similarly, we have  $2^{(d-1)-i}$  choices for $\delta_{\be_i}$  for $i=2, \dots, d-1$. Hence we have
\[
|\mathcal{Q}_D| = 2^d\cdot 2^{(d-2)}\cdot 2^{d-3} \cdots 2^1\cdot 2^0= 2^{d+(d-2)+\cdots+1}=2^{1+\binom{d}{2}}
\]
and $\dim \mathcal{Q}_D=1+\binom{d}2$ as desired.
\end{proof}

%\begin{notation}\label{eta}
%For each $\gamma\in \Z_2^{n}$, define $\tilde{\gamma}: D \to Z_2^{n}/ C$ by $\tilde{\gamma}(\be)=\gamma\cdot \be \mod C$. Then $\tilde{\gamma}\in \mathcal{S}$.
%Therefore, the map $\gamma \to \tilde{\gamma}$ induces a linear homomorphism from
%$\eta:\Z_2^{n} \to \mathcal{S}$ and the kernel of $\eta$ is given by
%\[
%Ker(\eta)=\mathcal{P}=\{\gamma \in \Z_2^{n}\mid \gamma\cdot \be \in C \text{ for all } \be\in D\}.
%\]
%\end{notation}

\begin{lemma}\label{Leta} For $\gamma\in\Z_2^n$, the map $\eta(\gamma)$ :$D\to \Z_2^n/D^\perp$, $\beta\mapsto \gamma\cdot\beta+D^\perp$ belongs to $\mathcal{Q}_D$.
\end{lemma}
\begin{proof} Since the coordinatewise product $\cdot$ is $\Z_2$-linear, so is $\eta(\gamma)$.
For $\beta\in D$,
$\langle\eta(\gamma)(\beta),\mathbf{1}+\beta\rangle=\langle\gamma\cdot\beta,\mathbf{1}+\beta\rangle=0$. Hence
$\eta(\gamma)\in\mathcal{Q}_D$.
\end{proof}

\begin{lemma}\label{ga}
Let $D$ be a $d$-dimensional triply even code of length $n$ divisible by $16$. Assume that $D$ contains $\mathbf{1}$ and that
$\dim (D\cdot D)=\binom{d}{2}+1$.
%Assume that there exists a basis $B=\{\mathbf{1}, \be_1, \dots, \be_{d-1}\}$ of $D$ such that $ \{\mathbf{1}\} \cup \{ \be\cdot \be'\mid \be, \be'\in B, \be \neq \be'\}$ is linearly independent.
%\begin{enumerate}[{\rm (1)}]
%\item
Then, for $\delta\in \mathcal{Q}_D$, there exists $\gamma\in \Z_2^{n}$ such that $\delta(\be)= \gamma \cdot \be\mod D^\perp$ for any $\be\in D$.
%\item For $\gamma\in D$ with $\gamma\neq\mathbf{1}$, $\{\alpha\cdot \gamma\mid \alpha\in D^\perp\}=\{\alpha\in \Z_2^n\mid {\rm wt}(\alpha)\in2\Z,\ {\rm Supp}(\alpha)\subset{\rm Supp}(\gamma)\}$.
%\end{enumerate}
\end{lemma}

\begin{proof} %(1):
Set $\Imm (\eta)=\{\eta(\gamma)\mid\gamma\in \Z_2^n\}$ and $\Ker(\eta)=\{\gamma\in\Z_2^n\mid \gamma\cdot\beta\in C \text{ for all } \beta\in D\}$.
By Lemma \ref{Leta}, it suffices to prove that $\dim \mathcal{Q}_D = \dim \Imm (\eta)$.
%, where $\eta$ is defined as in
%Notation \ref{eta}.  Recall that
%\[
%\Ker(\eta)=\mathcal{P}=\{\gamma \in \Z_2^{48}\mid \gamma\cdot \be \in C \text{ for all } \be\in D\}.
%\]
%Thus,
Since
\[
\begin{split}
\gamma\in \Ker(\eta)  & \Leftrightarrow  \langle\gamma\cdot \be, \be'\rangle=0\quad  \text{ for all } \be, \be'\in D\\
& \Leftrightarrow  \langle\gamma,  \be\cdot \be'\rangle=0\quad  \text{ for all } \be, \be'\in D,
\end{split}
\]
%we have ${\rm Span}_{\Z_2}\{ \be\cdot\be'|\be, \be'\in D\}=\Ker(\eta)^\perp$.
we have $D\cdot D=\Ker(\eta)^\perp$.
By the assumption, we have
%By the assumption on $B$, $\dim {\rm Span}_{\Z_2}\{ \be\cdot\be'\mid\be, \be'\in D\}=1+\binom{d}{2}$.
%Thus,
\[
\dim \Imm(\eta) = n -\dim \Ker(\eta) = \dim \Ker(\eta)^\perp=1+\binom{d}{2}.
\]
%By the assumption on $B$, $\dim P'=1+\binom{d}{2}$.
Therefore by Lemma \ref{dS}, $\dim \Imm(\eta)=\dim\mathcal{Q}_D$.
%(2): Set $E=\{\gamma\cdot\alpha\mid \alpha\in D^\perp\}$.
%Clearly, $E\subset \{\alpha\in \Z_2^n\mid {\rm wt}(\alpha)\in2\Z,\ {\rm Supp}(\alpha)\subset{\rm Supp}(\gamma)\}$.
%Hence, it suffices to show that $\dim E=\wt(\gamma)-1$.
%Let $B'$ be a basis of $D$ such that $\1,\1+\gamma\in B'$.
%Then $\{\1\}\cup\{\beta\cdot\beta'\mid \beta,\beta'\in B'\}$ is a basis of ${\rm Span}_{\Z_2}\{\alpha\cdot\alpha'\mid\alpha,\alpha'\in B\}$.
%By the assumption of $B$, its dimensin is $1+\binom{d}{2}$.
%Hence $\{\beta\cdot\beta'\mid \beta,\beta'\in B'\}$ is linear independent.
%Set $F=\{\alpha\in D^\perp\mid {\rm Supp}(\alpha)\subset{\rm Supp}(\1+\gamma)\}$.
%Then $\dim F+\dim E=\dim D^\perp$.
%Moreover, $F=\{\alpha\in\Z_2^n\mid {\rm Supp}(\alpha)\subset{\rm Supp}(\1+\gamma), (\alpha,\beta)=0\ \forall \beta\in B'\setminus\{\1+\gamma\}\}$
%Since $\{(\1+\gamma)\cdot\beta\mid \beta\in B',\ \beta\neq\1+\gamma\}$ is linear independent, we have $\dim F=\wt (\1+\gamma)-(d-1)=n-\wt(\gamma)-d+1$.
%Hence $\dim E=(n-d)-\dim F=\wt(\gamma)-1$.
%Let $B=\{\mathbf{1}, \be_1, \dots, \be_{d-1}\}$ be a basis of $D$. Then by Lemma \ref{LID},
%the set $$ \{\mathbf{1}\} \cup \{ \be\cdot \be'\mid \be, \be'\in B, \be \neq \be'\}$$  is linearly independent. Moreover, it spans $\mathcal{P}'$.  Therefore, $\dim \mathcal{P}'= 1+\binom{d}{2} = \dim \mathcal{S}$ and we have the desired conclusion.
\end{proof}

\begin{lemma}\label{delta} Let $V=\bigoplus_{\be\in D}V^\be$ and $U=\bigoplus_{\be\in D}U^\be$ be holomorphic framed VOAs with the same $\sfr{1}{16}$-code $D$.
Let $C=D^\perp$.
Then
there exists a unique $\delta\in \mathcal{Q}_D$ such that, as $V_C$-modules,
\[
U^\be\cong M_C(0,\delta(\beta))\fusion_{V_C} V^\be \quad \text{ for all }\be\in D.
\]
\end{lemma}

\begin{proof} Recall that $V^0\cong U^0\cong V_C$.
Let $\beta\in D$.
Then by Theorems \ref{PFrameSCE} and \ref{thm:4.3}, Lemma \ref{lem:4.6} and Remark \ref{rem:4.6}, there exist unique  $\gamma_{\beta,V},\gamma_{\beta,U}\in\Z_2^n/C$ such that $U^\beta\cong M_C(\beta,\gamma_{\beta,U})$ and $V^\beta\cong M_C(\beta,\gamma_{\beta,V})$ as $V_C$-modules.
Let $\delta$ be the map from $D$ to $\Z_2^n/C$ defined by $\delta(\beta)=\gamma_{\beta,U}+\gamma_{\beta,V}$.
%By Lemma \ref{lem:4.6}, $M_C(\beta,\gamma)\cong M_C(\beta,c+\gamma)$ as $V_C$-modules for any $c\in C$.
%Hence we may regard $\delta$ as the map from $D$ to $\Z_2^n/C$.
Then by Lemma \ref{lem:4.12} $$M_C(0,\delta(\beta))\fusion_{V_C}V^\beta\cong U^\beta.$$
%Since $\tau(U^\be)=\tau(V^\be) =\be$ for each $\be \in D$,  there is $\delta_\be\in \Z_2^n$ such that
%\[
%U^\be \cong M_{\delta_\be+C}\fusion V^\be.
%\]

Let us show that $\delta\in\mathcal{Q}_D$.
Since both $U$ and $V$ are simple current extensions (Theorem \ref{PFrameSCE}), we have $U^\be\fusion_{V_C} U^{\be'}\cong U^{\be+\be'}$ and $V^\be\fusion_{V_C} V^{\be'}\cong V^{\be+\be'}$ for all $\beta,\beta'\in D$.
Hence
\[
\delta({\be})+\delta({\be'})= \delta({\be+\be'}) \text{  for all } \be, \be'\in D,
\]
that is, the map $\delta: D \to \Z_2^n/C$ is $\Z_2$-linear. Moreover, $U^\be $ and $V^\be$ have integral weights for all $\be\in
D$. By Lemma \ref{lem:4.12}, the difference of their top weights is
$\langle\delta(\beta),\delta(\beta)+\beta\rangle/2=\langle\delta(\beta),\mathbf{1}+\beta\rangle/2$ mod $\Z$. Hence $\langle
\delta({\be}), \mathbf{1} + \be\rangle=0$ for all $\be\in D$. Thus $\delta \in \mathcal{Q}_D$. \end{proof}

%\begin{remark} In general, $\delta$ is not unique (Lemma \ref{lem:4.6}).
%\end{remark}

\begin{theorem}\label{TUni} Let $D$ be a $d$-dimensional triply even code of length $n$ divisible by $16$.
% containing $\1$.
%Assume that there exists a basis $B=\{\mathbf{1}, \be_1, \dots, \be_{d-1}\}$ of $D$ such that $ \{\mathbf{1}\} \cup \{ \be\cdot \be'\mid \be, \be'\in B, \be \neq \be'\}$ is linearly independent.
%Assume that $\dim (D\cdot D)=\binom{d}{2}+1$.
%Let $D$ be a subcode of $D_{[10]}$ and $C=D^\perp$.
Assume that $D$ contains $\mathbf{1}$ and that $\dim (D\cdot D)=\binom{d}{2}+1$. Let $U$ and $V$ be holomorphic framed
VOAs with the same $\sfr{1}{16}$-code $D$. Then $U\cong V$ as VOAs.
\end{theorem}

\begin{proof} Set $C=D^\perp$.
Let $V=\bigoplus_{\be\in D}V^\be$ and $U=\bigoplus_{\be\in D}U^\be$.
Note that $V^0\cong U^0\cong V_C$.
By Lemma \ref{delta},
there exists $\delta\in \mathcal{Q}_D$ such that
\[
U^\be\cong M_C(0,\delta(\be))\fusion_{V_C} V^\be \quad \text{ for all }\be\in D
\]
as $V_C$-modules.
By Lemma \ref{ga}, there exists $\gamma\in\Z_2^n$ such that $\delta(\be) =\gamma \cdot \be \mod C$ for all $\beta\in D$.
By Lemma \ref{Lsigma} we have
\[
U^\be\cong M_C(0,{\gamma\cdot \be}) \fusion_{V_C} V^\be \cong \sigma_\gamma\circ V^\be
\]
as $V_C$-modules for all $\be \in D$. Hence $\sigma_\gamma \circ V\cong U$ as $V_C$-modules.
By the uniqueness of simple current extensions (Proposition \ref{PSCE}), $\sigma_\gamma \circ V\cong U$ as VOAs.
The theorem follows since $V\cong \sigma_\gamma \circ V$ as VOAs.
\end{proof}

Combining Proposition \ref{LID} and Theorem \ref{TUni}, we obtain the following corollary:

\begin{corollary}\label{MC} For a subcode $D$ of the exceptional triply even code $D^{ex}$ of length $48$, the isomorphism class of a framed VOA of central charge $24$ with the $\sfr{1}{16}$-code $D$ is uniquely determined.
\end{corollary}

%\begin{remark}
%Let $V=\oplus_{\be\in D} V^\be$ be a framed VOA with the $\sfr{1}{16}$-code $D$ and $V^0=V_C$.
%In \cite{LY}, it was shown that if $\gamma \in \mathcal{P}=\{\gamma \in \Z_2^{48}\mid \gamma\cdot \be \in C \text{ for all } \be\in D\}$, then the $\sigma$-involution $\sigma_{\gamma}\in V_C$ will lift to an automorphism of $V$.
%\end{remark}

\section{Isomorphisms of holomorphic framed VOAs of central charge $24$ associated to quadratic spaces} \label{sec:qs}
In this section, we discuss isomorphisms between holomorphic VOAs of central charge $24$ associated to some maximal totally singular subspaces.
%prove that the isomorphism class of a holomorphic framed VOA associated to a quadratic space is uniquely
%determined by the Lie algebra structure.

\subsection{Quadratic subspaces and maximal totally singular subspaces}\label{sec:2.4}
First, we review a classification  of maximal totally singular subspaces up to certain equivalence from \cite{LS}. For the notation
and the detail, see \cite[Section 4]{LS}.

%Let $R$ be a $2m$-dimensional vector space over $\F_2$.
%A form $\langle\cdot,\cdot\rangle:R\times R\to \F_2$ is said to be {\it symplectic} if it is a symmetric bilinear form with $\langle a,a\rangle=0$ for all $a\in R$.
%A map $q:R\to\F_2$ is called a {\it quadratic} form associated to $\langle\cdot,\cdot\rangle$ if $\langle a,b\rangle=q(a+b)+q(a)+q(b)$ for all $a,b\in R$, and the pair
%Let $q$ be a quadratic form.
%$(R,q)$ is called a {\it quadratic space} if $q$ is a quadratic form on $R$.
%We often denote it by $R$ simply.
%A quadratic space $(R,q)$ is said to be {\it non-singular} if $R^\perp=\{a\in R\mid \langle a,R\rangle=0\}=0$.
%A vector $a\in R$ is said to be {\it singular} and {\it non-singular} if $q(a)=0$ and $q(a)=1$, respectively.
%A subspace $S$ of $R$ is said to be {\it totally singular} if any vector in $S$ is singular.
%A quadratic form $q$ is said to be {\it of plus type} and {\it of minus type} if the dimension of a maximal totally singular subspace of $(R,q)$ is $m$ and $m-1$, respectively.
%Let $O(R,q)$ denote the orthogonal group of $(R,q)$, the subgroup of ${\rm GL}(R)$ preserving $q$.

Let $(R,q)$ be a $2m$-dimensional quadratic space of plus type over $\F_2$. Then $(R^3,q^3)$ is a $6m$-dimensional quadratic
space of plus type over $\F_2$, where $q^3:R^3\to\F_2$, $q^3(v_1,v_2,v_3)=\sum_{i=1}^3 q(v_i)$.

\medskip

Consider the following
condition on maximal totally singular subspaces $\mathcal{S}$ of $R^3$:
\begin{eqnarray}
(a_1,a_2,0),(0,a_2,a_3)\in\mathcal{S}\ \text{for some}\ a_i\in R\setminus\{0\}\ {\rm with}\ q(a_i)=0.\label{Cond2}
\end{eqnarray}
We will recall the construction of certain maximal totally singular subspaces of $R^3$ not satisfying (\ref{Cond2}) from \cite{LS}.

\begin{theorem}\label{TClassify}{\rm (\cite[Theorem 4.6]{LS})} Let $S_1$ be a $k_1$-dimensional totally singular subspace of $R$ and let $S_2$ be a $k_2$-dimensional totally singular subspace of $S_1$.
Assume that $m-k_1-k_2$ is even. Let $P$ be an $(m-k_1-k_2)$-dimensional non-singular subspace of $S_1^\perp$ of
$\varepsilon$ type, where $\varepsilon\in\{\pm\}$. Let $Q$ and $T$ be complementary subspaces of $S_1$ and of $S_2$ in
$(S_1\perp P)^\perp$ and in $(S_2\perp P)^\perp$, respectively.
%Let $U=Q^\perp$.
Then the following hold:
\begin{enumerate}[{\rm (1)}]
\item $T$ and $Q^\perp$ are non-singular isomorphic quadratic spaces;
\item Let $\varphi$ be an isomorphism of quadratic spaces from $T$ to $Q^\perp$ and set
%let %$\mathcal{S}(S_1,S_2,P,Q,T,\varphi)$ be the subspace of $R^3$ defined by
\begin{eqnarray*}
%{\rm Span}_{\F_2}\biggl\{
%(s_1+p+q,s_2+p+t,q+\varphi(t)) \big|\ s_i\in S_i,p\in P,q\in Q,t\in T\biggr\}.
\mathcal{S}(S_1,S_2,P,Q,T,\varphi)=\{(s_1+p+q,s_2+p+t,q+\varphi(t)) |\ s_i\in S_i,p\in P,q\in Q,t\in T\}.
\end{eqnarray*}
Then $\mathcal{S}(S_1,S_2,P,Q,T,\varphi)$ is a maximal totally singular subspace of $R^3$;
\item $\mathcal{S}(S_1,S_2,P,Q,T,\varphi)$ depends only on $k_1,k_2$ and $\varepsilon$ up to $O(R,q)\wr\Sym_3$.
%\item $\mathcal{S}(S_1,S_2,P,Q,T,\varphi)$ does not satisfy (\ref{Cond2}).
\end{enumerate}
\end{theorem}

\begin{notation}
By (3), we denote $\mathcal{S}(S_1,S_2,P,Q,T,\varphi)$ by $\mathcal{S}(m,k_1,k_2,\varepsilon)$.
\end{notation}

\begin{theorem}\label{TClassify2}{\rm (\cite[Theorem 4.8]{LS})} Let $S_1$ be a $k_1$-dimensional totally singular subspace of $R$ and let $S_2$ be a $k_2$-dimensional totally singular subspace of $S_1$.
Assume that $m-k_1-k_2$ is odd. Let $P$ and $Q$ be $(m-k_1-k_2-1)$-dimensional and $(m-k_1+k_2-1)$-dimensional
non-singular subspaces of $S_1^\perp$ and of $(S_1\perp P)^\perp$ of plus type, respectively. Let $B$ and $T$ be
complementary subspaces of $S_1$ and of $S_2$ in $(S_1\perp P\perp Q)^\perp$ and in $(S_2\perp P\perp B)^\perp$,
respectively. Let $U=(Q\perp B)^\perp$. Then the following hold:

\begin{enumerate}[{\rm (1)}]
\item $B$ is a $2$-dimensional non-singular subspace of plus type;
\item $T$ and $U$ are isomorphic non-singular quadratic spaces of plus type;
\item Let $y$ be the non-singular vector in $B$ and let $z$ be a non-zero singular vector in $B$.
Let $\varphi$ be an isomorphism of quadratic spaces from $T$ to $U$ and set {\small
\begin{eqnarray*}
&&\mathcal{S}(S_1,S_2,P,Q,B,T,z,\varphi)=\\
&&\biggr\langle(s_1+p+q,s_2+p+t,q+\varphi(t)),(y,y,0),(y,0,y),(z,z,z)
%\\&&\hspace{2cm}
\bigg|\ s_i\in S_i, p\in P,q\in Q,t\in T\biggl\rangle_{\F_2}.
\end{eqnarray*}}
Then $\mathcal{S}(S_1,S_2,P,Q,B,T,z,\varphi)$ is a maximal totally singular subspace of $R^3$;
\item $\mathcal{S}(S_1,S_2,P,Q,B,T,z,\varphi)$ depends only on $k_1,k_2$ up to $O(R,q)\wr\Sym_3$.
%\item $\mathcal{S}(S_1,S_2,P,Q,B,T,z,\varphi)$ does not satisfy (\ref{Cond2}).

\end{enumerate}
\end{theorem}

\begin{notation}
By (4), we denote $\mathcal{S}(S_1,S_2,P,Q,B,T,z,\varphi)$ by $\mathcal{S}(m,k_1,k_2)$.
\end{notation}

In \cite{LS}, maximal totally singular subspaces of $R^3$ were classified.

\begin{theorem}\label{TC}{\rm (\cite[Theorem 5.11]{LS})} Let $\mathcal{S}$ be a maximal totally singular subspace of $R^3$.
Then, up to $O(R,q)\wr\Sym_3$, one of the following holds:
\begin{enumerate}[{\rm (1)}]
%\item $\mathcal{S}$ satisfies (\ref{Cond1});
\item $\mathcal{S}$ satisfies (\ref{Cond2});
\item $\mathcal{S}$ is conjugate to $\mathcal{S}(S_1,S_2,P,Q,T,\varphi)$ defined as in Theorem \ref{TClassify};
\item $\mathcal{S}$ is conjugate to $\mathcal{S}(S_1,S_2,P,Q,B,T,z,\varphi)$ defined as in Theorem \ref{TClassify2}.
\end{enumerate}
\end{theorem}

\subsection{Holomorphic VOAs $\mathfrak{V}(\mathcal{S})$}\label{sec:5}
Next  we review some facts about the VOA $\mathfrak{V}(\mathcal{S})$ defined in \cite{Sh6,LS}.
Throughout this subsection, $V$ denotes the VOA $V_{\sqrt{2}E_8}^+$.

Let $R(V)$ be the set of isomorphism classes of irreducible $V$-modules.
Then under the fusion rules, $R(V)$ forms an elementary abelian $2$-group of order $2^{10}$ (\cite{ADL,Sh2}).
Consider the map $q_V:R(V)\to\F_2$ defined by setting $q_V([M])=0$ and $1$ if $M$ is $\Z$-graded and is $(\Z+1/2)$-graded, respectively.
Then $(R(V),q_V)$ is a $10$-dimensional quadratic space of plus type over $\F_2$ (\cite[Theorem 3.8]{Sh2})
and $(R(V)^3,q_V^3)$  is a  $30$-dimensional quadratic space of plus type over $\F_2$.

\begin{notation}\label{nota:5.1}
Let $\mathcal{T}$ be a subset of $R(V)^3$.
We set $\mathfrak{V}(\mathcal{T})=\bigoplus_{[M]\in\mathcal{T}}M$ and often view it as a $V^{\otimes3}$-module by identifying  $R(V^{\otimes 3})$ with $R(V)^3$ (cf.\ \cite[Section 4.7]{FHL}).
\end{notation}

\begin{proposition}\label{PS1}{\rm (\cite[Proposition 4.4]{Sh6})} Let $\mathcal{T}$ be a subset of $R(V)^3$.
Then the $V^{\otimes 3}$-module $\mathfrak{V}(\mathcal{T})=\bigoplus_{[M]\in\mathcal{T}}M$
has a simple VOA structure of central charge $24$ by extending its $V^{\otimes3}$-module
structure if and only if $\mathcal{T}$ is a totally singular subspace of
$R(V)^3$. Moreover, $\mathfrak{V}(\mathcal{T})$ is holomorphic if and only if
$\mathcal{T}$ is maximal.
\end{proposition}

\begin{remark}\label{Rem} (\cite[Section 5]{LS})
\begin{enumerate}[{\rm (1)}]
\item A VOA is isomorphic to $\mathfrak{V}(\mathcal{T})$ for some totally singular subspace $\mathcal{T}$ of $R(V)^3$ if and only if it contains a full subVOA isomorphic to $V^{\otimes3}$.
\item If totally singular subspaces $\mathcal{T}_1$ and $\mathcal{T}_2$ of $R(V)^3$ are conjugate under $O(R(V),q_V)\wr\Sym_3$, then the VOAs $\mathfrak{V}(\mathcal{T}_1)$ and $\mathfrak{V}(\mathcal{T}_2)$ are isomorphic.
\end{enumerate}
\end{remark}

\begin{lemma}\label{PS2} {\rm (\cite[Lemma 5.4]{LS})} Let $\mathcal{S}$ be a maximal totally singular subspace of $R(V)^3$.
If $\mathcal{S}$ satisfies (\ref{Cond2}), then $\mathfrak{V}(\mathcal{S})$ is isomorphic to $V_L$ or its
$\Z_2$-orbifold $\tilde{V}_L$ for some even unimodular lattice $L$ of rank $24$ containing $(\sqrt{2}E_8)^{\oplus3}$.
Moreover, if $\mathcal{S}$ contains non-zero vectors $(a_1,0,0)$, $(0,a_2,0)$ and $(0,0,a_3)$ then $\mathfrak{V}(\mathcal{S})\cong V_L$ for a lattice with the same properties.
\end{lemma}

\subsection{Conjugacy classes of involutions in the automorphism group of $V_L$}\label{sec:4}
In this subsection, we discuss the conjugacy classes of certain involutions in $\Aut V_L$ when $L$ is the Niemeier lattice
$N(A_{15}D_9)$ or $N(A_7^2D_5^2)$. Throughout this subsection, let $L(\Phi)$ denote the root lattice  of a root system $\Phi$.

First, we summarize a few facts about lattices.
%For a root system $\Phi$, let $R(\Phi)$ denote a root lattice of type $\Phi$.

\begin{lemma}\label{LD52} Let $s$ be a root in $D_5$ and
%Let $\{s_i\mid 1\le i\le 5\}$ be a set of simple roots of $D_5$ and let $s_0$ be the highest root.
let $2\beta+L(D_5)$ be the order $2$ element in $L(D_5)^*/L(D_5)$.
Then $s+4\beta+2L(D_5)$ is conjugate to $s+2L(D_5)$ under the Weyl group of $D_5$.
\end{lemma}
\begin{proof} Let $\{e_i\mid 1\le i\le 5\}$ be an orthonormal basis of $\R^5$.
Then $\{\pm(e_i+ e_j),\pm(e_i-e_j)\mid 1\le i<j\le 5\}$ is a root system of type $D_5$, and $2\beta+D_5=e_1+D_5$.
Hence one can easily prove this lemma.
%Hence
%One can easily see that $s_0+4\beta+2D_5=s_1+2D_5$.
%Hence this assertion follows from the transitivity of the Weyl group on the set of roots.
\end{proof}

By \cite[p438, XVII]{CS}, we obtain the following lemma.

\begin{lemma}\label{LNA72} Let $N=N(A_7^2D_5^2)$ and $R=L(A_7^2D_5^2)$.
%Let $\alpha$ and $\beta$ be generators of $L(A_7)^*/L(A_7)$ and $L(D_5)^*/L(D_5)$ respectively.
Let $\tau$ be a diagram automorphism of $L(A_7)$.
%Then the following hold:
\begin{enumerate}[{\rm (1)}]
\item There exist generators $\alpha\in  L(A_7)^*/L(A_7)$ and $\beta\in L(D_5)^*/L(D_5)$ such that $N=\langle s,t,R\rangle_\Z$, where $s=(3\alpha,\alpha,\beta,0)$ and $t=(2\alpha,0,-\beta,\beta)$;
%\item The group $N(A_7^2D_5^2)/L(A_7^2D_5^2)$ is generated by $s=(3\alpha,\alpha,\beta,0)$ and $t=(2\alpha,0,-\beta,\beta)$;
\item The automorphism $(x_1,x_2,x_3,x_4)\mapsto(\tau(x_1),\tau(x_2),x_3,x_4)$ of $R^*$ does not preserve $N$.
\end{enumerate}
\end{lemma}

Next, we recall the following from \cite[Proposition 8.1, Exercise 10 in Chapter 8]{Kac}:

\begin{lemma}\label{LK1} Let $\g$ be a finite dimensional simple Lie algebra and let $g$ and $h$ be automorphisms of $\g$ of order $2$.
Assume that the fixed point subalgebras of $\g$ for $g$ and $h$ are isomorphic.
Then there exists an inner automorphism $x$ of $\g$ such that $xgx^{-1}=h$.
\end{lemma}

The next two lemmas follow from explicit calculations based on \cite[Chapter 8]{Kac}.
For a root system $\Phi$, let $\g(\Phi)$ denote the semi-simple Lie algebra of type $\Phi$.

\begin{lemma}\label{LD5}
\begin{enumerate}[{\rm (1)}]
\item Let $s$ be a root in $D_5$ and let $f=\exp ({\rm ad}(\pi\sqrt{-1}s))$, where we view $s$ as a vector in the Cartan subalgebra.
%Let $\{s_i\mid 1\le i\le 5\}$ be a set of simple roots of $D_5$ and let $s_0$ be the highest root.
%If an involution $f$ in $\Aut(\g(D_5))$ satisfies $\g(D_5)^f\cong\g(A_3A_1^2)$ then it is conjugate to $\exp ({\rm ad}(\pi\sqrt{-1}s))$ by an inner automorphism of $\g(D_5)$.
Then $\g(D_5)^f\cong\g(A_3A_1^2)$.
\item Let $g\in\Aut \g(D_5)$ be an involution which is a lift of the $-1$-isometry of $D_5$.
Then $\g(D_5)^g\cong\g(B_2^2)$.
\end{enumerate}
\end{lemma}
%\begin{proof}
%\end{proof}

\begin{lemma}\label{Lfixed} Let $\g_1\cong\g_2\cong\g(A_7)$ and $\g_3\cong\g_4\cong\g(D_5)$ and set $\g=\oplus_{i=1}^4\g_i$.
Let $f$ be an involution in $\Aut \g$ such that $\g^f\cong\g(A_7A_3B_2^2A_1^2)$.
Then the following hold:
\begin{enumerate}[{\rm (1)}]
\item $f(\g_1)=\g_2$, and $f(\g_i)=\g_i$ for $i=3,4$.
\item As sets of isomorphism classes, $\{\g_3^f,\g_4^f\}=\{\g(B_2^2),\g(A_3A_1^2)\}$.
\end{enumerate}
\end{lemma}

In the following,  we will show that the conjugacy classes of some involutions in $V_L$ are uniquely determined by the isomorphism class of the fixed point Lie subalgebra of $(V_L)_1$ for $L\cong N(A_{15}D_9)$ and $N(A_7^2D_5^2)$.
For a Lie algebra $\g$, let $\Inn \g$ denote the inner automorphism group of $\g$.
Since $\Inn (V_L)_1$ can be extended to an automorphism group of $V_L$, we view it as a subgroup of $\Aut V$.

\begin{theorem}\label{TA15D9} There exists exactly one conjugacy class of involutions $g$ in $\Aut V_{N(A_{15}D_9)}$ such that the fixed point Lie subalgebra $(V_{N(A_{15}D_9)}^g)_1$ is isomorphic to $\g(C_8B_4^2)$.
\end{theorem}
\begin{proof} Set $V=V_{N(A_{15}D_9)}$ and $\g=V_1$.
Let $g$ and $h$ be involutions in $\Aut V$ satisfying the assumption.
Since Cartan subalgebras of arbitrary Lie algebra are conjugate under inner automorphisms and any automorphism of finite order preserves a Cartan subalgebra (\cite[Lemma 8.1]{Kac}), we may assume that both $g$ and $h$ preserve the Cartan subalgebra $\C\otimes_{\Z}L(A_{15}D_9)$ of $\g$.
It follows from $\g\cong\g(A_{15})\oplus\g(D_9)$ that both $g$ and $h$ preserve each ideal.
By Lemma \ref{LK1}, there exists $x\in\Inn \g\subset\Aut V$ such that $g=xhx^{-1}$ on $\g$.

Set $k=xhx^{-1}g^{-1}$.
Then $k$ is trivial on $\g$.
Set $N=N(A_{15}D_9)$ and $R=L(A_{15}D_9)$.
Since $V_{R}$ is generated by $\g$ as a VOA, $k$ is also trivial on $V_{R}$.
By Schur's lemma $k$ acts on each irreducible $V_R$-submodule $V_{\lambda+R}$ of $V$ by a scalar.
Hence there exists $v\in 2R^*/2N$ such that $k=\exp(\pi\sqrt{-1}v_{(0)})$.
By \cite[p439, XIX]{CS}, we may assume that $N/R$ is generated by $(2\alpha,\beta)$, where $L(A_{15})^*/L(A_{15})=\langle \alpha\rangle$ and $L(D_9)^*/L(D_9)=\langle\beta\rangle$.
Then the group $2R^*/2N$ is generated by $(2\alpha,0)$.
We now consider the action of $g$ on $R^*\subset\C\otimes_{\Z}R$.
It follows from $\g^g\cong\g(C_8B_4^2)$ that ${g}(\alpha)=-\alpha$ and ${g}(\beta)=-\beta$ (cf.\ \cite[Proposition 8]{Kac}).
Hence ${g}(v)=-v$, and $gk^{1/2}g^{-1}=k^{-1/2}$, where $k^{1/2}=\exp(\pi\sqrt{-1}v_{(0)}/2)\in\Aut V$.
Thus we obtain $$k^{-1/2}xhx^{-1}k^{1/2}=k^{-1/2}kgk^{1/2}=g,$$
which proves the theorem.
\end{proof}

\begin{theorem}\label{TA7D5} There exists exactly one conjugacy class of involutions $g$ in $\Aut V_{N(A_7^2D_5^2)}$ such that the fixed point Lie subalgebra $(V_{N(A_7^2D_5^2)}^g)_1$ is isomorphic to $\g(A_7A_3B_2^2A_1^2)$.
\end{theorem}
\begin{proof}
Set $V=V_{N(A_7^2D_5^2)}$ and $\g=V_1$.
Then $\g\cong\g(A_7^2D_5^2)$, and let $\g_1,\g_2,\g_3,\g_4$ be ideals of $\g$ such that $$ \g=\bigoplus_{i=1}^4\g_i,\quad \g_1\cong\g_2\cong\g(A_7),\quad \g_3\cong\g_4\cong\g(D_5).$$

Let $g$ and $h$ be involutions in $\Aut V$ satisfying the assumption.
Since Cartan subalgebras of arbitrary Lie algebra are conjugate under inner automorphisms and any automorphism of finite order preserves a Cartan subalgebra (\cite[Lemma 8.1]{Kac}), we may assume that $g$ and $h$ preserve the Cartan subalgebra $\C\otimes_{\Z}L(A_{7}^2D_5^2)$ of $\g$.
By Lemma \ref{Lfixed}, we may assume that
$$g(\g_1)= h(\g_1)=\g_2,\quad \g_3^{g}\cong\g_3^{h}\cong\g(B_2^2),\quad \g_4^{g}\cong\g_4^{h}\cong\g(A_3A_1^2).$$
By Lemma \ref{LK1}, there exists $x\in\Inn(\g_3\oplus\g_4)\subset\Aut V$ such that $xhx^{-1}=g$ on $\g_3\oplus\g_4$.

Set $h'=xhx^{-1}$.
Let us consider the actions of $g$ and $h'$ on $\g_1\oplus \g_2$.
By Lemma \ref{Lfixed} (1), $h'g^{-1}$ preserves both $\g_1$ and $\g_2$.
Set $a_i=(h'g^{-1})_{|\g_i}$ for $i=1,2$.
%and $b=(h'g^{-1})_{|\g_2}$.
Then $h'=a_1a_2g$ on $\g_1\oplus\g_2$.
Since the order of $h'$ is $2$, we have %so are the orders of $a_i$, and
$a_2=ga_1^{-1}g^{-1}$. Hence $h'=a_1ga_1^{-1}$ and $a_1^{-1}h'a_1=g$ on $\g_1\oplus\g_2$.
Suppose that $a_1$ is not inner.
Then there exists $c\in\Inn \g_1$ such that $c^{-1}a_1$ acts on the Cartan subalgebra $\C\otimes_{\Z}L(A_7)$ of $\g_1$ as a diagram automorphism.
Hence $(c^{-1}a_1)((c^{-1}a_1)^{-1})^{g}=c^{-1}h'cg^{-1}$ acts on the Cartan subalgebra $\C\otimes_{\Z}L(A_7^2D_5^2)$ of $\g$ as $(x_1,x_2,x_3,x_4)\mapsto(\tau(x_1),\tau(x_2),x_3,x_4)$.
Since $c^{-1}h'cg^{-1}\in\Aut V$, its restriction on $\C\otimes_{\Z}L(A_{7}^2D_5^2)$ preserves $N(A_{7}^2D_5^2)$, which contradicts Lemma \ref{LNA72} (2).
Thus $a_1$ is inner, and it can be extended to $a\in\Aut V$.
Note that $ah'a^{-1}=g$ on $\g_1\oplus\g_2$.
Since $x$ is trivial on $\g_1\oplus\g_2$ and $a$ is trivial on $\g_3\oplus\g_4$, we have $(ax)h(ax)^{-1}=g$ on $\g$.

Set $h''=(ax)h(ax)^{-1}$ and $k=h''g^{-1}$.
Set $N=N(A_7^2D_5^2)$ and $R=L(A_7^2D_5^2)$.
Then $k$ is trivial on $\g$, and so is on $V_{R}$.
By Schur's lemma $k$ acts on each irreducible $V_{R}$-submodule $V_{\lambda+R}$ of $V$ by a scalar.
Hence $k=\exp(\pi\sqrt{-1}v_{0})$ for some $v\in 2R^*/2N$.
Let $\alpha\in L(A_7)^*/L(A_7)$ and $\beta\in L(D_5)^*/L(D_5)$ given in Lemma \ref{LNA72} (1).
Then $2R^*/2N$ is generated by $u$ and $v$, where
$u=(0,2\alpha,0,0)$, $v=(0,0,0,2\beta)$.
Note that the orders of $u$ and $v$ are $8$ and $4$ in $2R^*/2N$, respectively.
We now consider the action of $g$ on $R^*\subset\C\otimes_{\Z}R$.
It follows from $\g^g=\g(A_7A_3B_2^2A_1^2)$ that ${g}(v_1,v_2,v_3,v_4)=(v_2,v_1,-v_3,v_4)$ (cf.\ \cite[Proposition 8.1]{Kac}).
Hence
\begin{equation}
{g}(u)=3u-v,\quad {g}(v)=v.\label{Eq:g}
\end{equation}

Let $n,m\in\Z$ such that $k=\exp(\pi\sqrt{-1}(nu+mv)_{(0)})$.
Since $h''=kg$ is of order $2$, we have
$$(kg)^2=kgkg=\exp(\pi\sqrt{-1}(4nu+(-n+2m)v)_{(0)})={\rm Id}.$$
Hence $n\in2\Z$.
By (\ref{Eq:g}) and $h''=kg$, we have $$\exp(\pi\sqrt{-1}(-nu_{(0)}/2))^{-1} h''\exp(\pi\sqrt{-1}(-nu_{(0)}/2))=\exp(\pi\sqrt{-1}((m+\frac{n}{2})v_{(0)})g.$$
Hence we may assume that $n=0$ and $k=\exp(\pi\sqrt{-1}mv_{(0)})$.

In order to complete the proof, it suffices to show that the involutions $\exp(\pi\sqrt{-1}mv_{(0)})g$ and $g$ are conjugate.
Since the order of $\exp(\pi\sqrt{-1}mv_{(0)})g$ is $2$, we have $m\equiv0\pmod 2$ by (\ref{Eq:g}).
Hence we may assume $m=2$.
By Lemmas \ref{LD5}, $g$ acts on $\g_4$ as $\exp(\pi\sqrt{-1}s_{(0)})$ for some root $s\in L(D_5)$ up to conjugation.
Then by Lemma \ref{LD52}, $\exp(\pi\sqrt{-1}s_{(0)})$ is conjugate to $\exp(\pi\sqrt{-1}(2v+s)_{(0)})$, which completes this theorem.
\end{proof}

\subsection{Isomorphisms of the VOAs $\mathfrak{V}(\mathcal{S})$}\label{Proof2}
In this subsection, we establish the isomorphisms between certain VOAs $\mathfrak{V}(\mathcal{S})$.
Throughout this subsection, $V$ denotes $V_{\sqrt2E_8}^+$.

Let $\mathcal{S}$ be a maximal totally singular subspace of $R(V)^3$.
%In this subsection, we prove that the VOA structure of $\mathfrak{V}(\mathcal{S})$ is uniquely determined by the Lie algebra structure of the weight $1$ subspace when $\mathfrak{V}(\mathcal{S})\cong \g(C_8F_4^2)$ or $\g(A_7C_3^2A_3)$.
%The aim of this subsection is to show the VOAs in Proposition \ref{PAim} (2) and (3) are isomorphic for each case.
We now recall the $\Z_2$-orbifolds of $\mathfrak{V}(\mathcal{S})$ from \cite[Section 4.7]{LS}.
Let $W\in R(V)^3\setminus \mathcal{S}$ with $q_V^3(W)=0$.
Let $\chi_{W}:\mathcal{S}\to\F_2$ be the linear character of $\mathcal{S}$ defined by $\chi_W(W')=\langle W,W'\rangle$.
Then $\chi_W$ induces an automorphism $g_W$ of $\mathfrak{V}(\mathcal{S})$ of order $2$ acting
on $M'$ by $(-1)^{\chi_W(W')}$ for $W'=[M']\in\mathcal{S}$.
The fixed point subspace and the $\Z_2$-orbifold associated to $g_W$ are given as follows:

\begin{proposition}\label{PZ2}{\rm (\cite[Proposition 4.4]{LS})} The fixed point subspace of $\mathfrak{V}(\mathcal{S})$ with respect to $g_W$ is $\mathfrak{V}(\mathcal{S}\cap W^\perp)$, and the $\Z_2$-orbifold of $\mathfrak{V}(\mathcal{S})$ associated to $g_W$ is given by $\mathfrak{V}(\langle W, \mathcal{S}\cap W^\perp\rangle_{\F_2})$.
\end{proposition}

\begin{remark}\label{Rem2} The $\Z_2$-orbifold of $\mathfrak{V}(\mathcal{S})$ associated to $g_W$ exists and the VOA structure is uniquely determined.
Hence if $g\in\Aut \mathfrak{V}(\mathcal{S})$ is conjugate to $g_W$, then the $\Z_2$-orbifolds of $\mathfrak{V}(\mathcal{S})$ associated to $g$ and $g_W$ are isomorphic.
%There are only two simple current extensions of $\mathfrak{V}(\mathcal{S}\cap W^\perp)$ $\mathfrak{V}(\mathcal{S})$ and
\end{remark}

%\begin{remark}\label{Rem2} Let $U$ be a VOA and let $g$ and $h$ be conjugated involutions in $\Aut U$.
%Assume that the $\Z_2$-orbifolds associated to $g$ and $h$ exist then they are isomorphic.
%\end{remark}

\subsubsection{Holomorphic VOAs with Lie algebra $\g(C_{8}F_{4}^2)$}

The aim of this subsection is to show that the VOAs $\mathfrak{V}(\mathcal{S}(5,3,0,-))$ and $\mathfrak{V}(\mathcal{S}(5,3,2,+))$ are obtained as the $\Z_2$-orbifolds of $V_{N(A_{15}D_9)}$ associated to conjugated involutions.
For the descriptions of $\mathcal{S}(5,k_1,k_2)$ and $\mathcal{S}(5,k_1,k_2,\varepsilon)$, see Theorems \ref{TClassify} and \ref{TClassify2}, respectively.
For the calculations in the Lie algebra $\mathfrak{V}(\mathcal{S})_1$, see \cite[Section 5]{LS}.

\begin{proposition}\label{P540} Let $\mathcal{S}=\mathcal{S}(5,4,0)$.
Let $b$ and $d$ be non-singular vectors in $B^\perp$ and in $T$.
Set $W=(b,d+z,0)$ and $\mathcal{T}=\langle \mathcal{S}\cap W^\perp,W\rangle_{\F_2}$.
\begin{enumerate}[{\rm (1)}]
\item The subspace $\mathcal{T}$ is conjugate to $\mathcal{S}(5,3,0,-)$ under $O(R(V)^3,q_V^3)$.
\item The Lie algebra $\mathfrak{V}(\mathcal{S}\cap W^\perp)_1$ is isomorphic to $\g(C_{8}B_{4}^2)$.
\end{enumerate}
\end{proposition}
\begin{proof} Let $a$ and $c$ be non-zero singular vectors in $S_1$ and in $T$ such that $\langle a,b\rangle=\langle c,d\rangle=1$, respectively.
By the description of $\mathcal{S}(5,4,0)$, we have \begin{eqnarray*}
\mathcal{S}\cap W^\perp&=&\biggl\langle (s,0,0),(a+y,y,0),(y,0,y),(z,z,z),(0,y+t',y+\varphi(t')),(0,t,\varphi(t))\\ && \ \bigg|\ s\in S_1\cap b^\perp,\ t\in T\cap d^\perp,\ t'\in c+T\cap d^\perp\biggr\rangle_{\F_2}.\end{eqnarray*}
Since $W$ is singular, $\mathcal{T}$ is maximal totally singular.
Moreover, $\mathcal{T}$ does not satisfy (\ref{Cond2}).
By $\dim (\mathcal{T}\cap \{(r,0,0)\mid r\in R(V)\})=3$, $\dim (\mathcal{T}\cap \{(0,r,0)\mid r\in R(V)\})=0$ and Theorem \ref{TClassify}, $\mathcal{T}$ is conjugate to $\mathcal{S}(5,3,0,\varepsilon)$.
Since the image of the first coordinate projection $\mathcal{T}\to R(V)$ is $(S_1\cap b^\perp)\perp \langle a+y,b\rangle_{\F_2}$ and both $a+y$ and $b$ are non-singular, we have $\varepsilon=-$.
Thus we obtain (1).

Set $\mathcal{U}=\mathcal{S}\cap W^\perp$ and $\mathcal{X}^{(1)}=\mathcal{X}\cap \{(0,u,v)\mid u,v\in R(V)\}$ for $\mathcal{X}=\mathcal{T},\mathcal{U}$.
Then by \cite[Proposition 5.31]{LS} $\mathfrak{V}(\mathcal{T}^{(1)})_1$ is an ideal of $\mathfrak{V}(\mathcal{T})_1$ and $\mathfrak{V}(\mathcal{T}^{(1)})_1\cong\g(C_8)$.
It follows from $\mathcal{T}^{(1)}=\mathcal{U}^{(1)}$ that $\mathfrak{V}(\mathcal{U}^{(1)})_1$ is an ideal of $\mathfrak{V}(\mathcal{U})_1$ isomorphic to $\g(C_8)$.
Set $$\mathcal{U}'=\langle(s,0,0),(a+y,y,0),(y,0,y),(z,z,z)\mid s\in S_1\cap b^\perp\rangle_{\F_2}.$$
Then $\mathfrak{V}(\mathcal{U})_1=\mathfrak{V}(\mathcal{U}^{(1)})_1\oplus\mathfrak{V}(\mathcal{U}')_1$, and $\mathfrak{V}(\mathcal{U}')_1$ is an ideal.
By \cite[Proposition 5.30]{LS}, $\mathfrak{V}(\mathcal{S}(5,3,0,+))_1\cong\g(B_{4}^2D_{8})$.
One can see that $\mathfrak{V}(\mathcal{U}')_1$ is isomorphic to the ideal $\g(B_4^2)$ of $\mathfrak{V}(\mathcal{S}(5,3,0,+))_1$.
Hence (2) holds.
\end{proof}

\begin{proposition}\label{P421} Let $S_1,S_2,S_3$ be totally singular subspaces of $R(V)$ such that $S_3\subset S_2\subset S_1$ and $\dim S_1=4$, $\dim S_2=2$ and $\dim S_3=1$.
Let $Q$ and $T$ be complementary subspaces of $S_1$ and of $S_2$ in $S_1^\perp$ and in $S_2^\perp$, respectively.
Set $U=(S_3\perp Q)^\perp$.
Let $\varphi$ be an isomorphism from $T$ to $U$.
Let $$\mathcal{S}=\{(s_1+q,s_2+t,s_3+q+\varphi(t))\mid s_i\in S_i,q\in Q,t\in T\}.$$
Let $b\in (Q\perp U)^\perp$ be a non-zero singular vector such that $b\notin S_3^\perp$.
Set $W=(b,0,b)$ and $\mathcal{T}=\langle \mathcal{S}\cap W^\perp,W\rangle_{\F_2}$.
\begin{enumerate}[{\rm (1)}]
\item The subspace $\mathcal{S}$ of $R(V)^3$ is maximal totally singular.
\item The VOA $\mathfrak{V}(\mathcal{S})$ is isomorphic to the lattice VOA $V_{N(A_{15}D_{9})}$.
\item The subspace $\mathcal{T}$ of $R(V)^3$ is conjugate to $\mathcal{S}(5,3,2,+)$ under $O(R(V)^3,q_V^3)$.
\item The Lie algebra $\mathfrak{V}(\mathcal{S}\cap W^\perp)_1$ is isomorphic to $\g(C_{8}B_{4}^2)$.
\end{enumerate}
\end{proposition}
\begin{proof} Since $\dim Q=2$ and $\dim T=6$, we have $\dim \mathcal{S}=15$.
By the definition of $\mathcal{S}$, it is totally singular.
Hence we have (1).

Take non-zero singular vectors $h_i$ in $S_i$ for $i=1,2,3$.
Then $(h_1,0,0),(0,h_2,0),(0,0,h_3)\in\mathcal{S}$.
By Lemma \ref{PS2}, $\mathfrak{V}(\mathcal{S})$ is a lattice VOA.
By $\dim\mathfrak{V}(\mathcal{S})_1=408$ (cf.\ \cite[Proposition 5.17]{LS}), we have $\mathfrak{V}(\mathcal{S})\cong V_{N(A_{15}D_9)}$.
Hence (2) holds.

By the direct calculation, we have
\begin{equation}\mathcal{S}\cap W^\perp=\{(s_1+s_3+q,s_2+t,s_3+q+\varphi(t))\mid s_1\in S_1\cap b^\perp, s_2\in S_2, s_3\in S_3, q\in Q, t\in T\}.\label{Eq:U}
\end{equation}
Since $W$ is singular, $\mathcal{T}$ is maximal totally singular.
Moreover $\mathcal{T}$ does not satisfy (\ref{Cond2}).
Set $\mathcal{T}^{(ij)}=\{(r_1,r_2,r_3)\in\mathcal{T}\mid r_i=r_j=0\}$.
Then $\dim \mathcal{T}^{(23)}=3$, $\dim\mathcal{T}^{(13)}=2$, $\mathcal{T}^{(12)}=0$, and by Theorem \ref{TClassify},
$\mathcal{T}$ is conjugate to $\mathcal{S}(5,3,2,+)$, which proves (3).

By (\ref{Eq:U}), $\mathfrak{V}(\mathcal{S}\cap W^\perp)_1$ is the direct sum of two ideals
\begin{eqnarray*}
&&\mathfrak{V}(\{(s_1+s_3+q,0,s_3+q)\mid q\in Q, s_1\in S_1\cap b^\perp, s_3\in S_3\})_1,\\
&&\mathfrak{V}(\{(0,s_2+t,\varphi(t))\mid s_2\in S_2, t\in T\})_1,
\end{eqnarray*}
and their dimensions are $136$ and $72$, respectively.
The former is also an ideal of $\mathfrak{V}(\mathcal{T})_1$, and hence it is isomorphic to $\g(C_{8})$.
One can see that the latter is isomorphic to the ideal $\g(B_4^2)$ of $\mathfrak{V}(\mathcal{S}(5,3,0,+))_1$ (cf.\ \cite[Proposition 5.30]{LS}).
Hence (4) holds.
\end{proof}

It was shown in \cite[Proposition 5.40]{LS} that $\mathfrak{V}(\mathcal{S}(5,4,0))\cong V_{N(A_{15}D_9)}$.
Combining Remark \ref{Rem2}, Theorem \ref{TA15D9}, Propositions \ref{PZ2}, \ref{P540} and \ref{P421}, we obtain the following theorem:

\begin{theorem}\label{MT2}
The VOAs $\mathfrak{V}(\mathcal{S}(5,3,0,-))$ and $\mathfrak{V}(\mathcal{S}(5,3,2,+))$ are isomorphic.
\end{theorem}

\subsubsection{Holomorphic VOAs with Lie algebra $\g(A_{7}C_{3}^2A_{3})$}
The aim of this subsection is to show that the VOAs $\mathfrak{V}(\mathcal{S}(5,2,0))$ and $\mathfrak{V}(\mathcal{S}(5,2,1,+))$ are obtained as the $\Z_2$-orbifolds of $V_{N(A_{7}^2D_5^2)}$ associated to conjugated involutions.
For the descriptions of $\mathcal{S}(5,k_1,k_2)$ and $\mathcal{S}(5,k_1,k_2,\varepsilon)$, see Theorems \ref{TClassify} and \ref{TClassify2}, respectively.
For the calculations in the Lie algebra $\mathfrak{V}(\mathcal{S})_1$, see \cite[Section 5]{LS}.

\begin{proposition}\label{LS52} Let $\mathcal{S}=\mathcal{S}(5,2,0)$.
Let $a$ and $b$ be non-zero singular vectors in $P$ and $Q$, respectively.
Set $W=(a+b,0,0)$ and $\mathcal{T}=\langle \mathcal{S}\cap W^\perp, W\rangle_{\F_2}$.
\begin{enumerate}[{\rm (1)}]
\item The VOA $\mathfrak{V}(\mathcal{T})$ is isomorphic to $V_{N(A_7^2D_5^2)}$.
\item The Lie algebra $\mathfrak{V}(\mathcal{S}\cap W^\perp)_1$ is isomorphic to $\g(A_{7}A_{3}B_2^2A_{1}^2)$.
\end{enumerate}
\end{proposition}
\begin{proof} Let $c\in P$ and $d\in Q$ be non-singular vectors satisfying $\langle a,c\rangle=\langle b,d\rangle=1$.
Then $$\mathcal{T}=\biggl\langle(s,t,\varphi(t)),(a,a,0),(b,0,b),(c+d,c,d),(y,0,y),(0,y,y)\ \bigg|\  s\in \langle S_1,a+b\rangle_{\F_2},t\in T\biggr\rangle_{\F_2}.$$
Since $\mathcal{T}$ contains $(a,a,0)$ and $(a,0,b)$, the VOA $\mathfrak{V}(\mathcal{T})$ satisfies (\ref{Cond2}).
Hence by Lemma \ref{PS2} it is isomorphic to a lattice VOA or its $\Z_2$-orbifold.
It follows from $\dim\mathfrak{V}(\mathcal{T})_1=216$ (cf.\ \cite[Proposition 5.17]{LS}) that $\mathfrak{V}(\mathcal{T})\cong V_{N(A_{7}^2D_{5}^2)}$ or $\tilde{V}_{N(A_{17}E_{7})}$.
Note that $(V_{N(A_{7}^2D_{5}^2)})_1\cong\g(A_7^2D_5^2)$ and $(\tilde{V}_{N(A_{17}E_{7})})_1\cong\g(D_9A_7)$.
Since the subspace $$\mathfrak{V}\left(\biggl\langle(0,a,b),(0,y,y),(0,t,\varphi(t))\ \bigg|\ t\in T\biggr\rangle_{\F_2}\setminus\biggl\langle(0,y,y),(0,y+a,y+b)\biggr\rangle_{\F_2}\right)_1$$ is
a $126$-dimensional ideal, we have $\mathfrak{V}(\mathcal{T})_1\cong\g(A_7^2D_5^2)$ and $\mathfrak{V}(\mathcal{T})\cong V_{N(A_7^2D_5^2)}$.
Hence (1) holds.

Let us determine the Lie algebra structure of $\g=\mathfrak{V}(\mathcal{S}\cap W^\perp)_1$.
It is easy to see that $$\mathcal{S}\cap W^\perp=\biggl\langle(s,t,\varphi(t)),(a,a,0),(b,0,b),(c+d,c,d),(y,0,y),(0,y,y)\ \bigg|\ s\in S_1,t\in T\biggr\rangle_{\F_2}.$$
Then the subspace $$\mathfrak{V}\left(\biggl\langle(0,y,y),(0,t,\varphi(t))\mid t\in T\biggr\rangle_{\F_2}\setminus\{(0,y,y)\}\right)_1$$
is a $63$-dimensional ideal of $\g$, and it is also an ideal of $\mathfrak{V}(\mathcal{S}(5,2,0))_1$ isomorphic to $\g(A_{7})$.
One can see that the other $41$-dimensional ideal
\begin{equation}
\mathfrak{V}\left(\biggl\langle (s,0,0),(a,a,0),(b,0,b),(c+d,c,d),(y,y,0),(y,0,y)\ \bigg|\ s\in S_1\biggr\rangle_{\F_2}\right)_1.\label{Lie41}
\end{equation}
is isomorphic to $\g(A_3B_2^2A_1^2)$.
For the detail, see Appendix A.1.
Hence (2) holds.
\end{proof}

\begin{proposition}\label{LS521+} Let $\mathcal{S}=\mathcal{S}(5,2,1,+)$ and let $a$ be a non-zero singular vector in $Q$.
Set $W=(a,0,0)$ and $\mathcal{T}=\langle \mathcal{S}\cap W^\perp, W\rangle_{\F_2}$.
\begin{enumerate}[{\rm (1)}]
\item The VOA $\mathfrak{V}(\mathcal{T})$ is isomorphic to $V_{N(A_7^2D_5^2)}$.
\item The Lie algebra $\mathfrak{V}(\mathcal{S}\cap W^\perp)_1$ is isomorphic to $\g(A_{7}A_{3}B_2^2A_{1}^2)$.
\end{enumerate}
\end{proposition}
\begin{proof} Set $Q'=Q\cap a^\perp$.
Then $$\mathcal{T}=\{(s_1+p+q,s_2+p+t,s_3+q+\varphi(t))\mid s_1\in \langle S_1,a\rangle_{\F_2}, s_2\in S_2, s_3\in\langle a\rangle_{\F_2},p\in P,q\in Q',t\in T\}.$$
Take a non-zero singular vector $h_2\in S_2$.
Then it follows from $(a,0,0), (0,h_2,0),(0,0,a)\in\mathcal{T}$ and Lemma \ref{PS2} that $\mathfrak{V}(\mathcal{T})$ is a lattice VOA.
By $\dim\mathfrak{V}(\mathcal{T})_1=216$ (cf.\ \cite[Proposition 5.17]{LS}), we have $\mathfrak{V}(\mathcal{T})\cong V_{N(A_7^2D_5^2)}$.
Hence (1) holds.

By direct calculation, we have $$\mathcal{S}\cap W^\perp=\biggl\langle(s_1+p+q,s_2+p+t,q+\varphi(t))\ \bigg|\ s_i\in S_i,p\in P,q\in Q', t\in T\biggr\rangle_{\F_2}.$$
Let us determine the Lie algebra structure of $\g=\mathfrak{V}(\mathcal{S}\cap W^\perp)_1$.
Take non-zero singular vectors $h_1\in S_1$ and $h_2\in S_2$.
Then by \cite[Lemma 5.19 (2)]{LS}, $\mathfrak{V}(\{(h_1,0,0),(0,h_2,0)\})_1$ is a Cartan subalgebra of $\g$.
Consider the root space decomposition of $\g$ with respect to the Cartan subalgebra.
Then it is easy to see that
\begin{eqnarray}
&&\mathfrak{V}(\langle (s_1+p+q,s_2+p,q)\mid s_i\in S_i, p\in P,q\in Q'\rangle_{\F_2}\setminus\{(h_1,0,0),(0,h_2,0)\})_1,\label{6Id}\\
&&\mathfrak{V}(\{ (0,s_2+t,\varphi(t))\mid t\in T,s_2\in S_2\}\setminus\{(0,h_2,0)\})_1\label{6Id2}
\end{eqnarray}
are mutually orthogonal root spaces and their dimensions are $32$ and $56$.
Since (\ref{6Id2}) is contained in $\mathfrak{V}(\mathcal{S}(5,2,1,+))_1$, it is a root space of $\g(A_7)$.
One can see that (\ref{6Id}) is a root space of $\g(A_3B_2^2A_1^2)$.
For the detail, see Appendix A.2.
Hence (2) holds.
\end{proof}

Combining Remark \ref{Rem2}, Theorem \ref{TA7D5}, Propositions \ref{PZ2}, \ref{LS52} and \ref{LS521+}, we obtain the following theorem:

\begin{theorem}\label{MT3}
The VOAs $\mathfrak{V}(\mathcal{S}(5,2,0))$ and $\mathfrak{V}(\mathcal{S}(5,2,1,+))$ are isomorphic.
\end{theorem}

\appendix\section{Explicit descriptions of ideals in Section \ref{Proof2} }
In this appendix, we describe the ideals defined in (\ref{Lie41}) and (\ref{6Id}) as a direct sum of simple ideals.
Let $e_1,e_2,\dots, e_8$ be an orthogonal basis of $\R^8$ such that $\langle e_i,e_j\rangle=2\delta_{ij}$.
Then
\begin{equation*}
E=\sum_{1\le i,j\le 8}\Z(e_i+e_j)+\Z\frac{1}{2}\sum_{i=1}^8e_i\label{sqrt2E8}
\end{equation*}
is isomorphic to $\sqrt2E_8$.
Note that $E^*=E/2$.

\subsection{Explicit description for the ideal in (\ref{Lie41})}
Set $$\mathcal{U}=\langle(s,0,0),(a,a,0),(b,0,b),(c+d,c,d),(y,y,0),(y,0,y)\mid s\in S_1\rangle_{\F_2}.$$
Then $\dim\mathfrak{V}(\mathcal{U})_1=41$.
The aim of this subsection is to see $\mathfrak{V}(\mathcal{U})_1\cong\g(A_3B_2^2A_1^2)$.

Up to conjugation, we may assume that $$S_1=\langle[V_E^-],[V_{e_1+E}^+]\rangle_{\F_2}, y=[V_{(e_1+e_2)/2+E}^+], a=[V_{(e_1+e_2+e_3+e_4)/2+E}^+], b=[V_{(e_1+e_2+e_5+e_6)/2+E}^+].$$
For the detail of irreducible $V_E^+$-modules, see \cite{FLM}.
Note that $\mathfrak{V}(\{(s,0,0)\mid s\in S_1\})_1={\rm Span}_\C\{e_i(-1), x(e_i)^\pm\mid 1\le i\le 8\}$, where $x(e_i)^\pm=e^{e_i}\pm \theta (e^{e_i})\in V_{e_1+E}^\pm$.
Then $\mathfrak{V}(\mathcal{U})_1$ is a direct sum of the following simple ideals:
\begin{align*}
&\C e_1(-1)\oplus \C e_2(-1)\oplus \bigoplus_{\varepsilon\in\{\pm\},i=1,2}\C x(e_i)^\varepsilon\oplus \mathfrak{V}(\{(y+s,y,0),(y+s,0,y),(0,y,y)\mid s\in S_1\})_1,\\
&\C e_3(-1)\oplus\C e_4(-1)\oplus \bigoplus_{\varepsilon\in\{\pm\},i=3,4}\C x(e_i)^\varepsilon\oplus\mathfrak{V}(\{(y+a+s,y+a,0)\mid s\in S_1\})_1,\\
&\C e_5(-1)\oplus\C e_6(-1)\oplus \bigoplus_{\varepsilon\in\{\pm\},i=5,6}\C x(e_i)^\varepsilon\oplus\mathfrak{V}(\{(y+b+s,0,y+b)\mid s\in S_1\})_1,\\
&\C e_7(-1)\oplus \C x(e_7)^+\oplus \C x(e_7)^-,\\
&\C e_8(-1)\oplus \C x(e_8)^+\oplus \C x(e_8)^-.
\end{align*}
Since their dimensions are $15$, $10$, $10$, $3$ and $3$, we have $\mathfrak{V}(\mathcal{U})_1\cong\g(A_3B_2^2A_1^2)$.

\subsection{Explicit description for the ideal in (\ref{6Id})}
By the arguments in the proof of Lemma \ref{LS521+}, the $64$-dimensional subalgebra $\mathfrak{V}(\{(0,s+t,\varphi(t))\mid s\in S_2,t\in T\})_1$ contains the $63$-dimensional ideal.
Let $H'$ be its $1$-dimensional ideal.
% of $\mathfrak{V}(\{(0,s_2,0),(0,t,\varphi(t))\mid s_2\in S_2,t\in T\})$.
Then by (\ref{6Id2}), $H'\subset\mathfrak{V}(\{(0,h_2,0)\})_1$, where $h_2$ is the non-zero vector in $S_2$.
Set $$\mathcal{U}=\{ (s_1+p+q,p+s_2,q)\mid s_i\in S_i, p\in P,q\in Q'\}\setminus\{(0,h_2,0)\}.$$

In this subsection, we show that $\mathfrak{V}(\mathcal{U})_1\oplus H'\cong\g(A_3B_2^2A_1^2)$.
Note that its dimension is $41$.
Let $p_0$ be the non-singular vector in $P$.
Take a non-singular vector $q_0\in Q'$.
Then the set of all non-singular vectors in $Q'$ is $\{q_0,q_0+a\}$.
Up to conjugation, we may assume that $S_1=\langle [V_E^-],[V_{e_1+E}^+]\rangle_{\F_2}$, $S_2=\{ [V_E^\varepsilon]\mid \varepsilon\in\{\pm\}\}$, $p_0=[V_{(e_1+e_2)/2+E}^+]$, $q_0=[V_{(e_3+e_4)/2+E}^+]$, $a=[V_{(e_3+e_4+e_5+e_6)/2+E}^+]$.
%, $b=[V_{(e_1+e_2+e_5+e_6)/2+L}^+]$.
Then $\mathfrak{V}(\{(s,0,0)\mid s\in S_1\})_1={\rm Span}_\C\{e_i(-1), x(e_i)^\pm\mid 1\le i\le 8\}$, where $x(e_i)^\pm=e^{e_i}\pm\theta(e^{e_i})\in V_{e_i+E}^\pm$.
One can see that $\mathfrak{V}(\mathcal{U})_1$ is a direct sum of the following simple ideals:
\begin{align*}
&\C e_1(-1)\oplus \C e_2(-1)\oplus \bigoplus_{\varepsilon\in\{\pm\},i=1,2}\C x(e_i)^\varepsilon\oplus \mathfrak{V}(\{(p_0+s_1,p_0+s_2,0)\mid s_i\in S_i\})_1\oplus H',\\
&\C e_3(-1)\oplus\C e_4(-1)\oplus \bigoplus_{\varepsilon\in\{\pm\},i=3,4}\C x(e_i)^\varepsilon\oplus\mathfrak{V}(\{(q_0+s_1,0,q_0)\mid s_1\in S_1\})_1,\\
&\C e_5(-1)\oplus\C e_6(-1)\oplus \bigoplus_{\varepsilon\in\{\pm\},i=5,6}\C x(e_i)^\varepsilon\oplus\mathfrak{V}(\{(q_0+a+s_1,0,q_0+a)\mid s_1\in S_1\})_1,\\
&\C e_7(-1)\oplus \C x(e_7)^+\oplus \C x(e_7)^-,\\
&\C e_8(-1)\oplus \C x(e_8)^+\oplus \C x(e_8)^-.
\end{align*}
Since their dimensions are $15$, $10$, $10$, $3$ and $3$, we have $\mathfrak{V}(\mathcal{U})_1\oplus H'\cong\g(A_3B_2^2A_1^2)$.


\begin{thebibliography}{100000}

%\bibitem[ABD04]{ABD}
%T.\ Abe, G.\ Buhl and C.\ Dong, Rationality, regularity, and $C_2$-cofiniteness, {\it Trans. Amer. Math. Soc.} {\bf 356} (2004), 3391--3402.

%\bibitem[AD04]{AD}
%T.\ Abe and C.\ Dong, Classification of irreducible modules for the vertex operator algebra $V\sp +\sb L$: general case. {\it J. Algebra} {\bf 273} (2004), 657--685

\bibitem[ADL05]{ADL}
T.\ Abe, C.\ Dong, and H.\ Li, Fusion rules for the vertex operator algebra $M(1)$ and $V\sp +\sb L$, {\it Comm. Math. Phys.}{\bf 253} (2005), 171--219.

\bibitem[BM12]{BM} K. Betsumiya and A. Munemasa, On triply even binary codes, {\it J. London Math. Soc.} {\bf 86} (2012), 1--16.

\bibitem[Bo86]{Bo}
R.E.\ Borcherds, Vertex algebras, Kac-Moody algebras, and the Monster, {\it Proc.\ Nat'l.\ Acad.\ Sci.\ U.S.A.} {\bf 83} (1986), 3068--3071.

%\bibitem[Ch97]{Ch}
%C.C.\ Chevalley, The algebraic theory of spinors and Clifford algebras,
%Springer-Verlag, Berlin, 1997

\bibitem[CS99]{CS}
J.H.\ Conway and N.J.A.\ Sloane, Sphere packings, lattices and groups, 3rd Edition, Springer, New York, 1999.

\bibitem[DGM96]{DGM}
L.\ Dolan, P.\ Goddard and P.\ Montague, Conformal field theories, representations and lattice constructions, {\it Comm. Math. Phys.} {\bf 179} (1996), 61--120.

\bibitem[Do93]{Do}
C.\ Dong, Vertex algebras associated with even lattices, {\it J. Algebra} {\bf 161} (1993), 245--265.

\bibitem[DGH98]{DGH}
C.\ Dong, R.L.\ Griess, and G.\ H$\ddot{\rm o}$hn, Framed vertex operator algebras, codes and Moonshine module, {\it Comm.\ Math.\ Phys.} {\bf 193} (1998), 407--448.


\bibitem[DGL07]{DGL} C.\ Dong, R.L.\ Griess and C.H.\ Lam, Uniqueness results for the moonshine vertex operator algebra, {\it Amer. J. Math.} {\bf 129} (2007), 583--609.

\bibitem[DM04a]{DM}
C.\ Dong and G.\ Mason, Rational vertex operator algebras and the effective central charge, {\it Int. Math. Res. Not.} (2004), 2989--3008.

\bibitem[DM04b]{DMb}
C.\ Dong and G.\ Mason, Holomorphic vertex operator algebras of small central charge, {\it Pacific J. Math.} {\bf 213} (2004), 253--266.



%\bibitem[DLM96]{DLM}
%C.\ Dong, H.\ Li and G.\ Mason, Simple currents and extensions of vertex operator algebras, {\it Comm. Math. Phys.} {\bf 180} (1996), 671--707.




%\bibitem[DM06]{DM4}
%C.\ Dong and G.\ Mason, Integrability of $C_2$-cofinite vertex operator algebras.  {\it Int. Math. Res.
%Not.} (2006), Art. ID 80468, 15 pp.

\bibitem[DMZ94]{DMZ}
C.\ Dong, G.\ Mason and Y.\ Zhu, Discrete series of the Virasoro algebra and the moonshine module, {\it Proc.\ Sympos.\ Pure Math.} {\bf 56} (1994), 295--316.



\bibitem[FHL93]{FHL}
I.\ Frenkel, Y.\ Huang and J.\ Lepowsky, On axiomatic approaches to vertex operator algebras and modules,  Mem. Amer. Math. Soc. {\bf 104} 1993.

\bibitem[FLM88]{FLM}
I.\ Frenkel, J.\ Lepowsky and A.\ Meurman, Vertex operator algebras and the Monster, Pure and Appl.\ Math., Vol.134, Academic Press, Boston, 1988.

\bibitem[FZ92]{FZ}
I.\ Frenkel and Y.\ Zhu, Vertex operator algebras associated to representations
of affine and Virasoro algebras, {\it Duke Math. J.} {\bf 66} (1992),
123--168.

\bibitem[HL95]{HL}
Y.-Z\ Huang and J.\ Lepowsky, A theory of tensor products for module categories for a vertex operator algebra. III, {\it J. Pure Appl. Algebra} {\bf 100} (1995), 141--171.

%\bibitem[Gr98]{Gr1}
%R.L.\ Griess, A vertex operator algebra related to $E\sb 8$ with automorphism group ${\rm O}\sp +(10,2)$, {\it Ohio State Univ. Math. Res. Inst. Publ.} {\bf 7} (1998), 43--58.

\bibitem[Ka90]{Kac}
V.G.\ Kac, Infinite-dimensional Lie algebras. Third edition. Cambridge University Press, Cambridge, 1990.

%\bibitem[KKM91]{KKM}
%M.\ Kitazume, T.\ Kondo and I.\ Miyamoto, Even lattices and doubly-even codes, {\it J. Math. Soc. Japan} {\bf 43}  (1991), 67--87.


%\bibitem[HK00]{HK}
%M.\ Harada and M.\ Kitazume, $Z_4$-code constructions for the Niemeier lattices and their embeddings in the Leech lattice, {\it European J. Combin.} {\bf 21} (2000) 473--485.

%\bibitem[HS]{HS}
%G.\ H\"ohn and N.R.\ Scheithauer, A generalized Kac-Moody algebra of rank $14$, preprint, arXiv:1009.5153.

%\bibitem[PS75]{PS}
%V.\ Pless and N.J.A. Sloane, On the classification and enumeration of self-dual codes, {\it J. Combinatorial Theory Ser. A} {\bf 18} (1975) 313--335

\bibitem[La11]{Lam}
C.H\ Lam, On the constructions of holomorphic vertex operator algebras of
central charge $24$, {\it Comm. Math. Phys.} {\bf 305} (2011), 153--198

\bibitem[LS12]{LS}
C.H.\ Lam and H.\ Shimakura, Quadratic spaces and holomorphic framed vertex operator algebras of central charge 24, {\it Proc. Lond. Math. Soc.} {\bf 104} (2012), 540--576.

\bibitem[LY08]{LY}
C.H\ Lam and H.\ Yamauchi, On the structure of framed vertex operator algebras
and  their pointwise frame stabilizers,  {\it Comm. Math. Phys.} {\bf 277}
(2008),  237--285.



\bibitem[Mi96]{M1}
   M. Miyamoto, Binary codes and vertex operator (super)algebras,
  {\it J. Algebra} {\bf 181} (1996), 207--222.

\bibitem[Mi98]{M2}
M. Miyamoto, Representation theory of code vertex operator algebra, {\it J. Algebra} {\bf 201} (1998), 115--150.

\bibitem[Mi04]{M3}
M. Miyamoto, A new construction of the Moonshine vertex operator algebra over the real number field, {\it Ann.\ of Math.} {\bf 159} (2004), 535--596.

%\bibitem[SY03]{SY}
%S.\ Sakuma and H.\ Yamauchi, Vertex operator algebra with two Miyamoto involutions generating $S\sb 3$, {\it J. Algebra} {\bf 267} (2003), 272--297.

\bibitem[Sc93]{Sc93}
A.N.\ Schellekens, Meromorphic $c=24$ conformal field theories, {\it Comm. Math. Phys.} {\bf 153} (1993), 159--185.

\bibitem[Sh04]{Sh2}
H.\ Shimakura, The automorphism group of the vertex operator algebra $V_L^+$ for an even lattice $L$ without roots, {\it J. Algebra} {\bf 280} (2004), 29--57.

%\bibitem[Sh06]{Sh3}
%H.\ Shimakura, The automorphism groups of the vertex operator algebras $V_L^+$: general case, {\it Math. Z.} {\bf 252} (2006), 849--862.

\bibitem[Sh11]{Sh6}
H.\ Shimakura, An $E_8$-approach to the moonshine vertex operator algebra, {\it J. London Math. Soc.} {\bf 83} (2011), 493--516.

%\bibitem[Zh96]{Z}
%  Y.Zhu, Modular invariance of characters of vertex
%  operator algebras, {\it J. Amer. Math. Soc.} {\bf 9}
%  (1996), 237--302.

\end{thebibliography}
\end{document}